\documentclass[11pt,reqno]{amsart}
\usepackage{amsmath,amsxtra,latexsym,amsthm,amssymb,amscd,pb-diagram}
\usepackage{mathrsfs,mathabx,amsfonts}
\usepackage[colorlinks=true,linkcolor=red,citecolor=blue]{hyperref}
\usepackage[margin=1in]{geometry}

\usepackage{bm}
\usepackage{color}
\usepackage{makecell,multirow,diagbox}
\usepackage{mathtools}
\usepackage[displaymath, mathlines]{lineno}
\usepackage{tikz}
\usepackage{pgfplots}
\usepackage{graphicx}

\pgfplotsset{compat=1.18}

\usepackage{graphicx}
\usepackage{epsfig} 
\usepackage{array}
\usepackage{enumitem}
\DeclareMathOperator{\supp}{supp}
\newtheorem{definition}{Definition}[section]
\newtheorem{theorem}{Theorem}[section]
\newtheorem{proposition}{Proposition}[section]
\newtheorem{lemma}{Lemma}[section]

\newtheorem{remark}{Remark}[section]

\newcommand{\ity}{\infty}

\newcommand{\tron}[1]{\left(#1\right)}

\begin{document} 
	\title[The damped wave system with mixed product-type nonlinearities]{Critical thresholds for the damped wave system with mixed product-type nonlinearities}
	\subjclass{26A15, 35A01, 35B33, 35L52}
	\keywords{Damped wave equation; Coupled system; Global existence; Blow-up; Critical curve}
	\maketitle
	\centerline{\scshape \textbf{Trung Loc Tang}}
	{\footnotesize
		\centerline{Deparment of Mathematics, Thang Long University}
		\centerline{Nghiem Xuan Yem, Hanoi, Vietnam}
		\centerline{Email: loctt@thanglong.edu.vn}}
	
	\centerline{\scshape \textbf{Tuan Anh Dao}}
	{\footnotesize
		\centerline{Faculty of Mathematics and Informatics, Hanoi University of Science and Technology}
		\centerline{No.1 Dai Co Viet road, Hanoi, Vietnam}
		\centerline{Email: anh.daotuan@hust.edu.vn}}
	
	\centerline{\scshape \textbf{The Anh Cung}}
	{\footnotesize
		\centerline{School of Mathematics and Computer Science, Hanoi National University of Education}
		\centerline{136 Xuan Thuy, Hanoi, Vietnam}
		\centerline{Email: anhctmath@hnue.edu.vn}}
	
	\begin{abstract}
		In this paper, we study the Cauchy problem for a coupled system of damped wave equations with mixed product-type nonlinearities. We have succeeded in determining the critical curve, formulated in terms of the power exponents of nonlinearities, for the global (in time) existence of small data mild solutions and the nonexistence of global mild solutions. Moreover, the global existence result is established by means of Schauder's fixed point theorem and Banach's contraction principle, whereas the nonexistence analysis relies on a new iteration scheme adapted to the product-type structure of nonlinearities. To the best of our knowledge, this work provides the first critical threshold results for such systems.
	\end{abstract}
	
	\tableofcontents
	
	\section{Introduction}\label{sec1}
	
	First, let us begin with the Cauchy problem for semi-linear damped wave equations as follows:
	\begin{equation}\label{damped_wave_system}
		\left\{
		\begin{aligned}
			&u_{tt}-\Delta u+u_t=|u|^p, && x\in\mathbb{R}^n, t>0,\\
			&(u,u_t)(0,x)=(u_0,u_1)(x),&&x\in \mathbb{R}^n.
		\end{aligned}
		\right.
	\end{equation}
	This equation has been extensively studied as a model for dissipative hyperbolic equations with nonlinear source terms. The presence of the damping term $u_t$ produces a strong dissipative effect, leading to a remarkable interplay between the hyperbolic propagation and the parabolic diffusion. Consequently, many interesting results for solutions to \eqref{damped_wave_system} and its corresponding equation including decay estimates, asymptotic profile, global existence, finite time blow-up and lifespan have been investigated, for example, in \cite{Mat1976,Ikehata2002,Ikehata2004,Radu2011,IkeOga2016,Ikehata2019}.  Among these contributions, the authors in \cite{Todorova2001} established that, in the $L^1$ setting of the initial data, the critical exponent is $p_{\mathrm{Fuj}}:=1+2/n$. This value play a role as the threshold between the global existence of small data solutions and finite time blow-up, even for arbitrarily small nontrivial initial data. In particular, it exactly coincides with the Fujita exponent for the semi-linear heat equation to highlight the diffusion phenomenon for the damped wave equation. As we know, the single equation naturally motivates to study of weakly coupled systems in which different components interact through nonlinear source terms. In this setting, the competition between the two nonlinearities gives a more complicated threshold structure, which is usually described by a critical curve rather than a single critical exponent. This result was investigated in \cite{NishiharaWakasugi,Sun2007,Takeda2009}, where the authors studied the following weakly coupled system of semi-linear damped wave equations:
	\begin{equation}\label{System_weakly_coupled_damped_wave}
		\left\{
		\begin{aligned}
			&u_{tt}-\Delta u+u_t=|v|^p,&&x\in \mathbb{R}^n,t>0,\\
			&v_{tt}-\Delta v+v_t=|u|^{q},&&x\in \mathbb{R}^n,t>0,\\
			&(u,u_t,v,v_t)(0,x)=(u_0,u_1,v_0,v_1)(x),&&x\in \mathbb{R}^n.
		\end{aligned}
		\right.
	\end{equation}
	Under suitable regularity assumptions, they established the global (in time) existence of solutions for sufficiently small initial data in an $L^1$-based energy space, provided that
	\[
	\Gamma(p,q):=\frac{\max\{p,q\}+1}{pq-1}<\frac{n}{2}.
	\]
	Conversely, under appropriate positivity assumptions on the initial data, finite time blow-up occurs when $\Gamma(p,q)\geq n/2$. Therefore, we can say from this observation that the critical curve for \eqref{System_weakly_coupled_damped_wave} is
	\begin{equation}
		\label{critical_curve_weakly_coupled}
		\Gamma(p,q)=n/2.
	\end{equation}
	Very recently, the authors in \cite{Dabbicco2025} considered the following $2\times 2$ system of $\sigma$-evolution equations with additive coupling nonlinearities:
	\begin{equation}\label{System_weakly_coupled_damped_sigma}
		\left\{
		\begin{aligned}
			&u_{tt}+(-\Delta)^{\sigma}u+(-\Delta)^{\theta_1}u_t= c_1|u|^{\alpha}+c_2|v|^{p},&&x\in \mathbb R^n,t>0,\\
			&v_{tt}+(-\Delta)^{\sigma}v+(-\Delta)^{\theta_2}v_t= c_3|u|^{q}+c_4|v|^{\beta},&&x\in \mathbb R^n,\\
			&(u,u_t,v,v_t)(0,x)=(0,u_1,0,v_1)(x),&&x\in \mathbb{R}^n,
		\end{aligned}
		\right.
	\end{equation}
	where $\sigma<n<2\sigma$, $\theta_1,\theta_2\in [\sigma/2,\sigma]$ and some positive constants $c_k$ with $k\in\{1,2,3,4\}$. They have demonstrated the global (in time) existence of small data solutions under suitable conditions describing the interplay among the four nonlinear powers. More precisely, there appear two distinct critical curves including the first one similar to \eqref{critical_curve_weakly_coupled} and the genuinely new another arising from the interaction between $\alpha,\beta$ and $p,q$. Without loss of generality, if $q\leq q_{\rm c}:=(n+\sigma)/(n-\sigma)$, then these critical curves are given by
	\[
	p=\frac{n+\sigma}{q(n-\sigma)-2\sigma}\quad\text{and}\quad\beta=\frac{q(n-\sigma)}{p(n-\sigma)-2\sigma}.
	\]
	
	Motivated by those above results, in the present paper we study the following system with mixed product-type nonlinearities:
	\begin{equation}\label{Main_system}
		\left\{
		\begin{aligned}
			&u_{tt}-\Delta u+u_t=|u|^{\alpha}|v|^p,&&x\in \mathbb{R}^n,t>0,\\
			&v_{tt}-\Delta v+v_t=|u|^{q}|v|^{\beta},&&x\in \mathbb{R}^n,t>0,\\
			&(u,u_t,v,v_t)(0,x)=(u_0,u_1,v_0,v_1)(x),&&x\in \mathbb{R}^n,
		\end{aligned}
		\right.
	\end{equation}
	where $p,q>0$ and $\alpha,\beta \ge 0$. Two nonlinear terms $|u|^{\alpha}$ and $|v|^{\beta}$ can be interpreted as the nonlinearity pertubations of each equation of \eqref{Main_system} in a comparison with the corresponding equation of \eqref{System_weakly_coupled_damped_wave}. From this essential difference, one can say that \eqref{System_weakly_coupled_damped_wave} and \eqref{System_weakly_coupled_damped_sigma} are well-known as weakly coupled systems, meanwhile, the main problem \eqref{Main_system} of our work can be considerd as a strongly coupled system. The present paper provides a unified analysis of \eqref{Main_system} in two different range of the exponent of nonlinearity pertubations, namely, $\alpha,\beta\in[0,1)$ and $\alpha,\beta\in[1,\infty)$. The point worth noticing is that these two range also exhibit different critical mechanisms. More precisely, in the former case the nonlinearity pertubations $|u|^{\alpha}$ and $|v|^{\beta}$ lead to a new critical curve in comparison with \eqref{critical_curve_weakly_coupled}, whereas the critical threshold in the latter case relates to the Fujita exponent instead, governed by the total exponent of nonlinearities. The main difficulty of this work comes from the structure of nonlinearities, which are only H\"older continuous at the origin if the exponents $p,q,\alpha,\beta<1$, and the associated nonlinear mapping is generally not locally Lipschitz. This means that the usual contraction argument cannot be directly applied in this case. To overcome this difficulty, the key point is based on employing Schauder's fixed point theorem (see more \cite{DrabekMilota,Evan}). The required compactness is obtained by combining the finite propagation speed property of solutions with compact Sobolev embeddings on bounded domains. We then remove the compact support assumption on the initial data through an approximation together with a diagonal compactness argument. By contrast, if $p,q,\alpha,\beta\geq 1$, then the nonlinearities are locally Lipschitz at the origin. This allows us to use Banach's contraction principle and, in particular, to obtain uniqueness of the small data mild solution. Among other things, we develop an iteration scheme adapted to the mixed product-type nonlinearities for the nonexistence part. Starting from pointwise lower bounds for the damped wave kernel, we recursively improve the lower bounds for two solution components by using iteration arguments for $k$ steps (see later, Lemma \ref{lem_matrix}) to derive infinite $L^2$-growth of solutions in finite time as $k\to \ity$. This approach not only yields the nonexistence of global (in time) mild solutions in the subcritical case, but also explains the different forms of the critical thresholds in the two cases above. \medskip
	
	\textbf{Notations:} Throughout the paper, we write $f\lesssim g$ if there exists a constant
	$C>0$ such that $f\leq Cg$. Moreover, we write $f\sim g$ whenever both $f\lesssim g$ and $g\lesssim f$ hold. A positive constant $C$ appearing in some estimates may be changed from line to line. Usually, all Lebesgue and Sobolev spaces are understood to be defined over $\mathbb{R}^n$. We denote by $\mathcal{C}^\infty_{\rm c}(\Omega)$, the space of smooth functions compactly supported in $\Omega$ and by $\mathcal{D}'(\Omega)$ its continuous dual or the space of distributions on $\Omega$. We use the notation ${\bf 1}_{\Omega}(x)$ as the indicator function of $\Omega$, that is, ${\bf 1}_{\Omega}(x)=1$ if $x\in \Omega$ and ${\bf 1}_{\Omega}(x)=0$ if $x\not \in\Omega$. For $\mu\in\mathbb{R}$, we set $[\mu]^+:=\max\{\mu,0\}$ and $\lceil\mu\rceil:=\min\{k\in\mathbb{Z}:k\geq\mu\}$. We use $\mathbf{C}:=(C,C)^{\top}$ to denote a constant vector in $\mathbb{R}^2$. For two matrices $M=(m_{ij})$ and $N=(n_{ij})$ of the same size, the inequality $M\leq N$ is understood componentwise, namely, $m_{ij}\leq n_{ij}$ for all $i$ and $j$. Finally, the notations $A\hookrightarrow B$ and $A\Subset B$ mean that $A$ is continuously embedded and compactly embedded into $B$, respectively.\medskip
	
	\textbf{Main results:} The first one concerns with the global (in time) existence of small data solutions.
	
	\begin{theorem}\label{Thr_1_Globalexistence}
		Let $p,q>0$ and $\alpha,\beta\in[0,1)$ such that $\min\{p+\alpha,q+\beta\}>1$. Set $r=\min\{2,p+\alpha,q+\beta\}$ and $\kappa_{r}=(n-1)|1/2-1/r|$. Assume that $1\leq n\leq3$ and
		\begin{equation}\label{con_G-N_Thr1}
			\max\{p+\alpha,q+\beta\}<3\quad \text{ if }n=3,
		\end{equation}
		together with the following condition:    \begin{equation}\label{Thr1globalexistence_Criticalcurve}
			\frac{1+ \max\{q-\alpha,p-\beta\}}{pq-(1-\alpha)(1-\beta)}< \frac{n}{2}.
		\end{equation}
		Then there exists a constant $\varepsilon_0>0$ such that for any small initial data $(u_0,u_1,v_0,v_1)\in \mathcal{D}_{r}:= \big(H^1\cap H^{\kappa_{r}}_{r} \cap L^1\big)^2 \times \big(L^2\cap L^1\big)^2$ satisfying $\|(u_0,u_1,v_0,v_1)\|_{\mathcal{D}_{r}}\leq\varepsilon_0$, the problem \eqref{Main_system} admits a global (in time) mild solution in the sense of Definition \ref{MildSol.Def} belonging to
		\begin{align*}
			(u,v)\in \left(\mathcal{C}([0,\infty); H^1 \cap L^{r})\cap \mathcal{C}^1([0,\infty);L^2)\right)^2.
		\end{align*}
		Moreover, the following estimates hold:
		\begin{align*}
			\|u(t,\cdot)\|_{\dot{H}^1} &\lesssim (1+t)^{-\frac{n}{4}-\frac{1}{2}+[\ell_u]^+} \|(u_0,u_1,v_0,v_1)\|_{\mathcal{D}_{r}},\\
			\|u(t,\cdot)\|_{L^{r}} &\lesssim (1+t)^{-\frac{n}{2}(1-\frac{1}{r})+[\ell_u]^+} \|(u_0,u_1,v_0,v_1)\|_{\mathcal{D}_{r}},\\
			\|u(t,\cdot)\|_{L^2} &\lesssim (1+t)^{-\frac{n}{4}+[\ell_u]^+}  \|(u_0,u_1,v_0,v_1)\|_{\mathcal{D}_{r}},\\
			\|u_t(t,\cdot)\|_{L^2} &\lesssim (1+t)^{-\frac{n}{4}-1+[\ell_u]^+}\|(u_0,u_1,v_0,v_1)\|_{\mathcal{D}_{r}},\\
			\|v(t,\cdot)\|_{\dot{H}^1} &\lesssim (1+t)^{-\frac{n}{4}-\frac{1}{2}+[\ell_v]^+} \|(u_0,u_1,v_0,v_1)\|_{\mathcal{D}_{r}},\\
			\|v(t,\cdot)\|_{L^{r}} &\lesssim (1+t)^{-\frac{n}{2}(1-\frac{1}{r})+[\ell_v]^+} \|(u_0,u_1,v_0,v_1)\|_{\mathcal{D}_{r}},\\
			\|v(t,\cdot)\|_{L^2} &\lesssim (1+t)^{-\frac{n}{4}+[\ell_v]^+}  \|(u_0,u_1,v_0,v_1)\|_{\mathcal{D}_{r}},\\
			\|v_t(t,\cdot)\|_{L^2} &\lesssim (1+t)^{-\frac{n}{4}-1+[\ell_v]^+}\|(u_0,u_1,v_0,v_1)\|_{\mathcal{D}_{r}},
		\end{align*}
		where
		$$\ell_u:= (1-\alpha)^{-1}\left(1-\frac{n}{2}(p+\alpha-1)\right) + \varepsilon \text{ and } \ell_v:= (1-\beta)^{-1}\left(1-\frac{n}{2}(q+\beta-1)\right) + \varepsilon$$
		with any sufficiently small positive constant $\varepsilon$.
	\end{theorem}
	\begin{remark}
		{\rm
			The restriction \eqref{con_G-N_Thr1} for $n=3$ appearing in Theorem \ref{Thr_1_Globalexistence} arises from using the compact embedding $H^1(\Omega)\Subset L^{2(\alpha_w+\beta_w)}(\Omega)$ on bounded domains $\Omega\subset\mathbb{R}^3$ in our proof (see later, Section \ref{sec2}).
		}
	\end{remark}
	Next, a blow-up result reads as follows.
	\begin{theorem}\label{Thr1_Blowup}
		Let $p,q>0$ and $\alpha,\beta\in[0,1)$ such that $\min\{p+\alpha,q+\beta\}>1$. Assume that $1\leq n\leq3$, $u_0=v_0=0$ and $u_1,v_1\in L^1\cap H^1$ satisfying
		\[
		M_u:=\int_{\mathbb{R}^n}u_1(x)\,dx>0,\qquad M_v:=\int_{\mathbb{R}^n}v_1(x)\,dx>0.
		\]
		If
		\begin{equation}\label{Thr1blowup_Criticalcurve}
			\frac{1+ \max\{q-\alpha,p-\beta\}}{pq-(1-\alpha)(1-\beta)}> \frac{n}{2},
		\end{equation}
		then there is no global (in time) mild solution of \eqref{Main_system} in the sense of Definition \ref{MildSol.Def}.
	\end{theorem}
	\begin{remark}
		{\rm
			From the conditions \eqref{Thr1globalexistence_Criticalcurve} and \eqref{Thr1blowup_Criticalcurve} in Theorems \ref{Thr_1_Globalexistence} and \ref{Thr1_Blowup}, respectively, we may claim that the threshold to separate existence and nonexistence of global mild solutions is described by
			\begin{equation}\label{crit_curve_sublinear}
				\frac{1+ \max\{q-\alpha,p-\beta\}}{pq-(1-\alpha)(1-\beta)}= \frac{n}{2}.
			\end{equation}  
			Especially, when $\alpha=\beta=0$, the gained threshold \eqref{crit_curve_sublinear} exactly coincides with the critical curve \eqref{critical_curve_weakly_coupled}. The figure \ref{fig:critical_curve_shift} illustrates the shifting of the critical curve when $\alpha,\beta$ have change.
		}
		
		\begin{figure*}[t]
			\centering
			
			\begin{tikzpicture}
				
				\begin{axis}[
					width=8.5cm,
					height=7.0cm,
					xmin=0, xmax=5.15,
					ymin=0, ymax=5.15,
					axis lines=left,
					xlabel={$p$},
					ylabel={$q$},
					xtick=\empty,
					ytick=\empty,
					clip=false,
					enlargelimits=false,
					axis line style={->},
					label style={font=\normalsize},
					every axis x label/.style={
						at={(current axis.right of origin)},
						anchor=west
					},
					every axis y label/.style={
						at={(current axis.above origin)},
						anchor=south
					}
					]
					
					\node[
					anchor=north east,
					font=\scriptsize
					]
					at (axis cs:0,0) {$0$};
					
					\pgfmathsetmacro{\pcA}{2/3}
					\pgfmathsetmacro{\pcB}{5/3}
					
					\draw[
					dashed,
					black!60,
					line width=0.5pt
					]
					(axis cs:\pcA,0) -- (axis cs:\pcA,5.15);
					
					\draw[
					dashed,
					black!60,
					line width=0.5pt
					]
					(axis cs:0,\pcA) -- (axis cs:5.15,\pcA);
					
					\node[
					anchor=north,
					font=\scriptsize
					]
					at (axis cs:\pcA,0)
					{$p=2/n$};
					
					\node[
					anchor=east,
					font=\scriptsize
					]
					at (axis cs:0,\pcA)
					{$q=2/n$};
					
					\draw[
					dashed,
					black!60,
					line width=0.5pt
					]
					(axis cs:\pcB,0) -- (axis cs:\pcB,5.15);
					
					\draw[
					dashed,
					black!60,
					line width=0.5pt
					]
					(axis cs:0,\pcB) -- (axis cs:5.15,\pcB);
					
					\node[
					anchor=north,
					font=\scriptsize
					]
					at (axis cs:\pcB,0)
					{$p=1+2/n$};
					
					\node[
					anchor=east,
					font=\scriptsize
					]
					at (axis cs:0,\pcB)
					{$q=1+2/n$};
					
					\addplot[
					black,
					line width=0.65pt,
					domain=1:1.6666667,
					samples=120,
					smooth
					]
					{(5/3)/(x-2/3)};
					
					\addplot[
					black,
					line width=0.65pt,
					domain=1.6666667:5,
					samples=120,
					smooth
					]
					{(1+(2/3)*(x+1))/x};
					
					\addplot[
					blue,
					line width=0.65pt,
					domain=0.92:1.4666667,
					samples=120,
					smooth
					]
					{(94/75)/(x-2/3)};
					
					\addplot[
					blue,
					line width=0.65pt,
					domain=1.4666667:5,
					samples=120,
					smooth
					]
					{(0.72+(2/3)*(x+0.9))/x};
					
					\addplot[
					red!75!black,
					line width=0.65pt,
					domain=0.80:1.1666667,
					samples=120,
					smooth
					]
					{(19/30)/(x-2/3)};
					
					\addplot[
					red!75!black,
					line width=0.65pt,
					domain=1.1666667:5,
					samples=120,
					smooth
					]
					{(0.3+(2/3)*(x+0.6))/x};
					
					\addplot[
					green!55!black,
					line width=0.65pt,
					domain=0.71:0.8666667,
					samples=120,
					smooth
					]
					{(29/150)/(x-2/3)};
					
					\addplot[
					green!55!black,
					line width=0.65pt,
					domain=0.8666667:5,
					samples=120,
					smooth
					]
					{(0.06+(2/3)*(x+0.3))/x};
					
					\draw[
					black,
					line width=0.65pt
					]
					(rel axis cs:0.60,0.92)
					--
					(rel axis cs:0.70,0.92);
					
					\node[
					anchor=west,
					font=\scriptsize
					]
					at (rel axis cs:0.72,0.92)
					{$(\alpha,\beta)=(0,0)$};
					
					\draw[
					blue,
					line width=0.65pt
					]
					(rel axis cs:0.60,0.84)
					--
					(rel axis cs:0.70,0.84);
					
					\node[
					anchor=west,
					font=\scriptsize
					]
					at (rel axis cs:0.72,0.84)
					{$(\alpha,\beta)=(0.2,0.1)$};
					
					\draw[
					red!75!black,
					line width=0.65pt
					]
					(rel axis cs:0.60,0.76)
					--
					(rel axis cs:0.70,0.76);
					
					\node[
					anchor=west,
					font=\scriptsize
					]
					at (rel axis cs:0.72,0.76)
					{$(\alpha,\beta)=(0.5,0.4)$};
					
					\draw[
					green!55!black,
					line width=0.65pt
					]
					(rel axis cs:0.60,0.68)
					--
					(rel axis cs:0.70,0.68);
					
					\node[
					anchor=west,
					font=\scriptsize
					]
					at (rel axis cs:0.72,0.68)
					{$(\alpha,\beta)=(0.8,0.7)$};
					
					\node[
					anchor=north,
					align=center,
					font=\small
					]
					at (rel axis cs:0.50,-0.1)
					{};
					
				\end{axis}
			\end{tikzpicture}
			
			\vspace{-3mm}
			
			\caption{
				Shifting the critical curve for several choices of
				$(\alpha,\beta)$.
			}
			\label{fig:critical_curve_shift}
		\end{figure*}
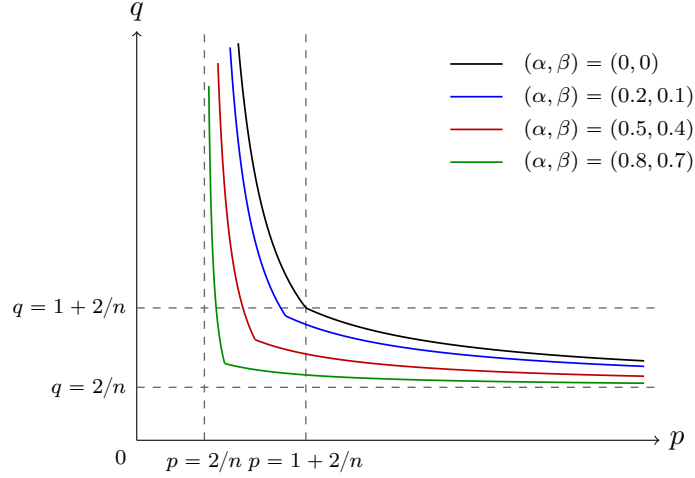
	\end{remark}
	\section{Global existence for small data solutions}\label{sec2}
	We first recall the representation formula for the solution to the linear Cauchy problem 
	\begin{equation}\label{Linear_damped_wave}
		\left\{
		\begin{aligned}
			&w_{tt}-\Delta w+w_t=0,&&x\in \mathbb{R}^n,t>0,\\
			&(w,w_t)(0,x)=(w_0,w_1),&&x\in \mathbb{R}^n,
		\end{aligned}
		\right.
	\end{equation}
	in the form
	\begin{equation*}
		w^{\rm lin}(t,x) = (\mathcal{K}(t,x) +\partial_t \mathcal{K}(t,x))\ast_x w_0(x) + \mathcal{K}(t,x) \ast_x w_1(x),
	\end{equation*}
	where the Fourier transform of $\mathcal{K}(t,x)$ is given by
	\begin{align*}
		\widehat{\mathcal{K}}(t,\xi) = 
		\displaystyle\frac{e^{-t/2}\sinh{\left(t \sqrt{1/4 -|\xi|^2}\right)}}{\sqrt{1/4- |\xi|^2}}  \text{ if } |\xi| < 1/2
		\,\,\text{ , }\,\, \widehat{\mathcal{K}}(t, \xi) = \displaystyle\frac{e^{-t/2}\sin{\left(t \sqrt{|\xi|^2-1/4}\right)}}{\sqrt{|\xi|^2-1/4}} \text{ if } |\xi| > 1/2
	\end{align*}
	and $\widehat{\mathcal{K}}(t,\xi) = t e^{-t/2}$ if $|\xi| = 1/2$. Let $\chi_k= \chi_k(r)$ with $k\in\{\rm L,H\}$ be smooth cut-off functions having the following properties:
	\begin{align*}
		&\chi_{\rm L}(r)=
		\begin{cases}
			1 &\quad \text{ if }r\le \varepsilon^*/2, \\
			0 &\quad \text{ if }r\ge \varepsilon^*,
		\end{cases}
		\text{ and } \qquad
		\chi_{\rm H}(r)= 1 -\chi_{\rm L}(r),
	\end{align*}
	where $\varepsilon^*$ is a sufficiently small constant. It is obvious to see that $\chi_{\rm H}(r)= 1$ if $r \geq \varepsilon^*$ and $\chi_{\rm H}(r)= 0$ if $r \le \varepsilon^*/2$. Now, we restate the following important result, which comes into play in our proof.
	
	\begin{proposition}\cite[Theorem 1.1]{IkedaInuiOkamoto2019}\label{LinearEstimates}
		Let $n\geq1$, $j\in\{0,1\}$, $1\leq m_1\leq m_2<\infty$, $m_2\neq1$,
		$\kappa:=(n-1)\left|1/2-1/m_2\right|$ and $s_1\geq s_2\geq0$.
		Then, the following estimate holds for all $t>0$:
		\begin{align*}
			&\big\|\partial_t^j |\nabla|^{s_1} \mathcal{K}(t,\cdot)\ast_x \varphi\big\|_{L^{m_2}}\\ &\quad\lesssim (1+t)^{-\frac{n}{2}(\frac{1}{m_1}-\frac{1}{r})-\frac{s_1-s_2}{2}-j} \big\||\nabla|^{s_2} \chi_{\rm L}(|\nabla|)\varphi\big\|_{L^{m_1}} + e^{-ct} \big\||\nabla|^{s_1}\chi_{\rm H}(|\nabla|) \varphi\big\|_{H_{m_2}^{\kappa+j-1}},
		\end{align*}
		where $c$ is a suitable positive constant.
		Moreover, the following estimate holds for $1\leq m_1\leq\infty$, $1<m_2<\infty$, and $d>n/m_2$:
		\begin{align*}
			&\big\|\partial_t^j |\nabla|^{s_1}\mathcal{K}(t,\cdot) \ast_x \varphi \big\|_{L^{\infty}}\\ &\quad\lesssim (1+t)^{-\frac{n}{2m_1}-\frac{s_1-s_2}{2}-j} \big\||\nabla|^{s_2} \chi_{\rm L}(|\nabla|) \varphi\big\|_{L^{m_1}} + e^{-ct}\big \||\nabla|^{s_1 + d} \chi_{\rm H}(|\nabla|) \varphi\big\|_{H^{\kappa+j -1}_{m_2}}.
		\end{align*}
	\end{proposition}
	For the brevity, we write $U=(u,v)$, $U_0=(u_0,v_0)$, $U_1=(u_1,v_1)$ and define
	\[
	F(U)=(F_u(U),F_v(U)):=(|u|^{\alpha}|v|^p,|u|^{q}|v|^{\beta}).
	\]  
	For convenience later, we denote
	\[
	(\alpha_u,\beta_u):=(\alpha,p),\quad
	(\alpha_v,\beta_v):=(q,\beta)\quad \text{ and }\quad F_w(U):=|u|^{\alpha_w}|v|^{\beta_w}
	\]
	with $w\in\{u,v\}$. To begin our proof, let us introduce the definition of a mild solution.
	\begin{definition}\label{MildSol.Def}
		Let $T\in(0,\infty]$. A pair $U=(u,v) \in(\mathcal{C}([0,T);L^2))^2$ is called a local (in time) mild solution of \eqref{Main_system} if $F_u(U),F_v(U)\in L_{\rm loc}^1\big([0,T);L^2\big)$ and the following relations hold true:
		\begin{equation}\label{duhamel_formula}
			\begin{cases}\vspace{0.1cm}
				u(t,x)=u^{\rm lin}(t,x)+\displaystyle\int_{0}^{t}\mathcal{K}(t-s,x)\ast_x F_u(U)\,ds, \\
				v(t,x)=v^{\rm lin}(t,x)+\displaystyle\int_{0}^{t}\mathcal{K}(t-s,x)\ast_x F_v(U)\,ds.
			\end{cases}
		\end{equation}
		If $T = \infty$, then we say that $(u,v)$ is a global (in time) mild solution to \eqref{Main_system}.
	\end{definition}
	We are going to give the proof of Theorem \ref{Thr_1_Globalexistence} with the assumption of compactly supported initial data in the first stage, and by removing the compact-support assumption in the second one through an approximation procedure combined with a compactness argument.

	\subsection{Proof of Theorem \ref{Thr_1_Globalexistence} for compactly supported initial data}
	In this section, let us assume that $(U_0,U_1)\in \mathcal{D}_{r}$ satisfying
	\[
	\supp u_j\cup\supp v_j\subset B_{R}:=\{x\in \mathbb{R}^n:|x|\leq R\},\qquad j\in\{0,1\}.
	\]
	Without loss of generality, we only give the proof for the case $p+\alpha\leq q+\beta$.
	
	\medskip\noindent {\bf $\bullet$ Step 1: Definition of the nonlinear mapping.} For each $t\geq 0$, we define
	$$ \Omega_t: = \{x\in \mathbb{R}^n:|x|\leq R+t\}. $$
	For each $T>0$, we define the space $X(T)$ as follows: 
	\begin{align*}
		X(T):=\Big\{U=(u,v)\in&\left(\mathcal{C}([0,T];L^2)\cap L^\infty([0,T]; H^1)\cap W^{1,\infty}([0,T];L^2)\right)^2\\
		&:\supp w(t,\cdot)\subset \Omega_t\text{ for any }t\in[0,T]\text{ and }w\in\{u,v\} \Big\}.
	\end{align*}
	The critical curve condition and $p+\alpha\leq q+\beta$ imply $q+\beta>1+2/n$, and hence $[\ell_v]^+=0$. For $w\in\{u,v\}$, setting
	\[
	\begin{aligned}
		\mathcal E_w(s):=&f_1(s)\|w(s,\cdot)\|_{L^{r}}
		+f_2(s)\|w(s,\cdot)\|_{L^2} +f_3(s)\|w(s,\cdot)\|_{\dot{H}^1}
		+f_4(s)\|\partial_tw(s,\cdot)\|_{L^2},
	\end{aligned}
	\]
	where
	\begin{align*}
		f_1(s)=(1+s)^{\frac{n}{2}\tron{1-\frac{1}{r}}},\;\;f_2(s)=(1+s)^{\frac{n}{4}},\;\;f_3(s)=(1+s)^{\frac{n+2}{4}},\;\;f_4=(1+s)^{\frac{n+4}{4}},
	\end{align*}
	we equip $X(T)$ with the weighted norm
	\[
	\|U\|_{X(T)}:=\operatorname*{ess\,sup}_{0\leq s\leq T}
	\left\{(1+s)^{-[\ell_u]^+}\mathcal E_u(s)+\mathcal E_v(s)\right\}.
	\]
	We define the mapping
	$$\mathcal{N}: U\in X(T)\mapsto \mathcal{N}[U]:=(\mathcal{N}_u[U],\mathcal{N}_v[U]), $$
	where
	\begin{equation}\label{mapN}
		\begin{aligned}
			\mathcal{N}_w[U](t,\cdot)&=(\mathcal{K}(t,\cdot)+\partial_t\mathcal{K}(t,\cdot))\ast_xw_0
			+\mathcal{K}(t,\cdot)\ast_xw_1
			+\int_0^t\mathcal{K}(t-s,\cdot)\ast_xF_w(U(s,\cdot))\,ds\\
			&=:w^{\rm lin}(t,\cdot)+w^{\rm nl}(t,\cdot)\qquad\text{ with } w\in\{u,v\}.
		\end{aligned}
	\end{equation} 
	
	\noindent {\bf $\bullet$ Step 2: Construction of a nonempty compact convex set.} Let us define
	\[
	\mathcal{B}_T:=\{W=(\phi,\psi)\in X(T):\|W\|_{X(T)}\leq \varepsilon\}
	\]
	and let
	\[
	E_T:=\left(\mathcal{C}([0,T];L^2(\Omega_T))\right)^2
	\]
	be equipped with the norm
	\[
	\|W\|_{E_T}:=\sup_{0\leq t\leq T}\Big\{\|\phi(t,\cdot)\|_{L^2(\Omega_T)}+\|\psi(t,\cdot)\|_{L^2(\Omega_T)}\Big\}.
	\] 
	\begin{lemma}
		$\mathcal{B}_T$ is convex and compact in $E_T$.
	\end{lemma}
	\begin{proof}
		The definition of $\mathcal{B}_T$ implies that $\mathcal{B}_T$ is bounded in $(L^\infty([0,T];H^1(\Omega_T)))^2$ and that the family $\{\partial_tW:W\in\mathcal{B}_T\}$ is bounded in $(L^{\infty}([0,T];L^2(\Omega_T)))^2$. By the Rellich-Kondrachov theorem, $H^1(\Omega_T)\Subset L^2(\Omega_T)$. Thanks to the Arzel\`a-Ascoli theorem, it implies $\mathcal{B}_T$ is relatively compact in $E_T$. Now we show that $\mathcal{B}_T$ is closed in $E_T$. Indeed, let $\{W^{(k)}\}_{k\in\mathbb N}\subset \mathcal{B}_T$ be a sequence such that $W^{(k)}\to W$ in $E_T$ as $k\to \ity$. Extending all functions by zero outside $\Omega_T$, we equivalently have
		\[
		W^{(k)}\to W
		\quad\text{in }
		\bigl(\mathcal{C}([0,T];L^2)\bigr)^2\text{ as }k\to \ity.
		\]
		Since $\{W^{(k)}\}_{k\in \mathbb N}$ is bounded in $\bigl(L^\infty([0,T];H^1)\bigr)^2$ and $\{\partial_tW^{(k)}\}_{k\in \mathbb N}$ is bounded in $\bigl(L^\infty([0,T];L^2)\bigr)^2$, one has
		\[
		W^{(k)}\stackrel{*}{\rightharpoonup} W
		\quad\text{in }\bigl(L^\infty([0,T];H^1)\bigr)^2
		\]
		and
		\[
		\partial_tW^{(k)}\stackrel{*}{\rightharpoonup} V\quad\text{in }\bigl(L^\infty([0,T];L^2)\bigr)^2
		\]
		as $k\to\ity$. Passing to the limit in the distributional identity
		$\partial_tW^{(k)}=dW^{(k)}/dt$, we obtain $V=\partial_tW$.
		Hence, it holds
		\[
		W\in\left(\mathcal{C}([0,T];L^2)\cap L^\infty([0,T];H^1)\cap W^{1,\infty}([0,T];L^2)\right)^2.
		\]
		Write $W^{(k)}=(\phi^{(k)},\psi^{(k)})$ and $W=(\phi,\psi)$. For any $t\in[0,T]$ and any component $w\in\{\phi,\psi\}$, it follows that
		\[
		\|w(t)\|_{L^2(\mathbb R^n\setminus\Omega_t)}
		\leq\|w(t)-w^{(k)}(t)\|_{L^2}\to0
		\]
		as $k\to\ity$, therefore, we derive $\supp w(t,\cdot)\subset\Omega_t$. Next, the weighted bounds defining $\mathcal{B}_T$ are preserved by weak lower semicontinuity. For the $L^r$ norm, we use
		\[
		\|w^{(k)}(t)-w(t)\|_{L^r}
		\leq |\Omega_T|^{\frac1r-\frac12}
		\|w^{(k)}(t)-w(t)\|_{L^2}
		\]
		to give $\|W\|_{X(T)}\leq\varepsilon$, and hence $W\in \mathcal{B}_T$. As a consequence, $\mathcal{B}_T$ is closed in $E_T$. Since $\mathcal{B}_T$ is relatively compact in $E_T$, we conclude that $\mathcal{B}_T$ is compact in $E_T$. Moreover, for $W_1,W_2\in \mathcal{B}_T$ and $\lambda\in [0,1]$, we have
		\[
		\|\lambda W_1+(1-\lambda)W_2\|_{X(T)}\leq \lambda \|W_1\|_{X(T)}+(1-\lambda)\|W_2\|_{X(T)}\leq \varepsilon,
		\]
		which is to verify the convexity of $\mathcal{B}_T$.
	\end{proof}
	
	\noindent {\bf $\bullet$ Step 3: Application of Schauder's fixed point theorem.} Let $U=(u,v)\in \mathcal{B}_T$. Since $U\in(L^\infty([0,T];H^1(\Omega_T)))^2$, the Sobolev embedding theorem and the condition \eqref{con_G-N_Thr1} imply that 
	\[
	F_w(U)=|u|^{\alpha_w}|v|^{\beta_w}\in L^\infty([0,T];L^2(\Omega_T))\subset L^1([0,T];L^2(\Omega_T)).
	\]
	Since $\supp (F_w(U(s)))\subset\Omega_s\subset \Omega_T$, we indentify $F_w(U)$ with its zero extension outside $\Omega_T$. Consequently, we have $F_w(U)\in L^1([0,T];L^2)$. Using Lemma \ref{regularity_K} with $g=|u|^{\alpha_w}|v|^{\beta_w}$, we obtain the required regularity for the Duhamel term. The finite propagation speed of the damped wave equation gives $\supp (U^{\rm lin}(t))\subset \Omega_t$ for $t\in [0,T]$, hence $\supp \mathcal{N}_w[U](t,\cdot)\subset\Omega_t\subset \Omega_T$ for $t\in [0,T]$, with $w\in\{u,v\}$. Thus, $\mathcal{N}[U]\in X(T)\subset E_T$ for any $U\in\mathcal{B}_T$. To apply Schauder's fixed point theorem, we need to prove that $N$ is continuous in $E_T$ and satisfies the following estimate:
	\begin{equation}\label{keyestimate_Thr1globalexistence}
		\|\mathcal{N}[U]\|_{X(T)}\leq C_0\big(\|(U_0,U_1)\|_{\mathcal{D}_{r}}+\|U\|^{p+\alpha}_{X(T)}+\|U\|^{q+\beta}_{X(T)}\big).
	\end{equation}
	
	\begin{lemma}\label{lemma_G-N_Thr1globalexistence}
		Under the assumptions of Theorem \ref{Thr_1_Globalexistence}, the following estimates hold:
		\begin{align*}
			\big\||u(s,\cdot)|^{\alpha_w}|v(s,\cdot)|^{\beta_w}\big\|_{L^{\gamma}} &\lesssim  (1+s)^{-\frac{n}{2}(\alpha_w+\beta_w-\frac{1}{\gamma})+\alpha_w[\ell_u]^+} \|U\|_{X(T)}^{\alpha_w+\beta_w}
		\end{align*}
		for $w\in\{u,v\}$ and $\gamma\in \{1,2,r\}$.
	\end{lemma}
	\begin{proof}
		The application of H\"older's inequality and Lemma \ref{fractionalGagliardoNirenberg} yields
		\begin{align*}
			\big\||u(s,\cdot)|^{\alpha_w}|v(s,\cdot)|^{\beta_w}\big\|_{L^{\gamma}}
			&\lesssim \|u(s,\cdot)\|^{\alpha_w}_{L^{\gamma(\alpha_w+\beta_w)}}\|v(s,\cdot)\|^{\beta_w}_{L^{\gamma(\alpha_w+\beta_w)}} \\
			&\lesssim \|u(s,\cdot)\|^{\alpha_w(1-\theta_{r,w})}_{L^{r}}\|u(s,\cdot)\|^{\alpha_w\theta_{r,w}}_{\dot{H}^1}\|v(s,\cdot)\|_{L^{r}}^{\beta_w(1-\theta_{r,w})} \|v(s,\cdot)\|_{\dot{H}^{1}}^{\beta_w\theta_{r,w}}  \\
			&\lesssim (1+s)^{-\frac{n}{2}\tron{\alpha_w+\beta_w-\frac{1}{\gamma}}+\alpha_w[\ell_u]^+}  \|U\|_{X(T)}^{\alpha_w+\beta_w}
		\end{align*}
		for $w\in\{u,v\}$, where $\theta_{r,w}$ is given by
		\[
		\theta_{r,w}=\frac{\frac{1}{r}-\frac{1}{\gamma(\alpha_w+\beta_w)}}{\frac{1}{r}-\frac{1}{2}+\frac{1}{n}}.
		\]
		The condition \eqref{con_G-N_Thr1} ensures that $\theta_{r,w}\in[0,1]$. This completes the proof of Lemma \ref{lemma_G-N_Thr1globalexistence}.
	\end{proof}
	\noindent Returning to the proof of \eqref{keyestimate_Thr1globalexistence}, the estimates for $u^{\rm lin}$ and $v^{\rm lin}$ have already been established in Proposition \ref{LinearEstimates}, so it remains to prove that
	\begin{equation}\label{eqNu_Thr1}
		\|(u^{\rm nl},v^{\rm nl})\|_{X(T)}\lesssim \|U\|_{X(T)}^{p+\alpha}+ \|U\|_{X(T)}^{q+\beta}.
	\end{equation}
	First, using Lemma \ref{lemma_G-N_Thr1globalexistence} and
	Proposition \ref{LinearEstimates} with $(m_1,m_2) = (1,2)$, $(s_1, s_2) =(1,0)$ for $\tau \in [0, t/2]$, and $(m_1,m_2)=(2,2), (s_1,s_2) =(1,0)$ for $\tau \in (t/2,t]$, we obtain
	\begin{align*}
		\| w^{\rm nl}(t,\cdot)\|_{\dot{H}^1} &\lesssim \int_0^{t/2} (1+t-s)^{-\frac{n+2}{4}} \big\||u(s,\cdot)|^{\alpha_w}|v(s,\cdot)|^{\beta_w}\big\|_{L^1\cap L^2}ds\\
		&\qquad + \int_{t/2}^t (1+t-s)^{-\frac{1}{2}} \big\||u(s,\cdot)|^{\alpha_w}|v(s,\cdot)|^{\beta_w}\big\|_{L^2}ds\\
		&\lesssim  (1+t)^{-\frac{n+2}{4}} \|U\|_{X(T)}^{\alpha_w+\beta_w} \int_0^{t/2}(1+s)^{-\frac{n}{2}\tron{\alpha_w+\beta_w-1}+\alpha_w[\ell_u]^+}ds\\
		&\qquad +  (1+t)^{-\frac{n}{2}\tron{\alpha_w+\beta_w-\frac{1}{2}}+\alpha_w[\ell_u]^+} \|U\|_{X(T)}^{\alpha_w+\beta_w} \int_{t/2}^t (1+t-s)^{-\frac{1}{2}}ds.
	\end{align*}
	We use the inequalities
	\begin{equation}\label{eq_lossofdecay}
		1-n(p+\alpha-1)/2+\alpha[\ell_u]^+<[\ell_u]^+
	\end{equation}
	and 
	\begin{equation}\label{eq_condition_gt_1}
		n(q+\beta-1)/2-q[\ell_u]^+>1,
	\end{equation}
	which follow from assumption \eqref{Thr1globalexistence_Criticalcurve}. Consequently, we arrive at
	\begin{equation}\label{eq_H1_Thr1}
		\begin{aligned}
			&\|u^{\rm nl}(t,\cdot)\|_{\dot{H}^1}\lesssim (1+t)^{-\frac{n+2}{4}+[\ell_u]^+}\|U\|^{p+\alpha}_{X(T)},\quad \|v^{\rm nl}(t,\cdot)\|_{\dot{H}^1}\lesssim (1+t)^{-\frac{n+2}{4}}\|U\|^{q+\beta}_{X(T)}.
		\end{aligned}
	\end{equation}
	Next, we estimate $\|u^{\rm nl}(s,\cdot)\|_{L^{\gamma}}$ and  $\|v^{\rm nl}(s,\cdot)\|_{L^{\gamma}}$ for $\gamma\in\{2,r\}$. Since $\kappa_{r}<1$, applying Lemma \ref{lemma_G-N_Thr1globalexistence} together with Proposition \ref{LinearEstimates} with $(m_1,m_2) = (1,\gamma), (s_1, s_2) = (0,0)$ for $\tau \in [0, t/2)$ and $(m_1,m_2)=(\gamma,\gamma), (s_1, s_2) = (0,0)$ for $\tau\in[t/2,t]$, we obtain
	\begin{align*}
		\|w^{\rm nl}(t,\cdot)\|_{L^{\gamma}} &\lesssim \int_0^{t/2} (1+t-s)^{-\frac{n}{2}\tron{1-\frac{1}{\gamma}}} \big\||u(s,\cdot)|^{\alpha_w}|v(s,\cdot)|^{\beta_w}\big\|_{L^1\cap L^{\gamma}} ds\\
		&\qquad+\int_{t/2}^t \big\||u(s,\cdot)|^{\alpha_w}|v(s,\cdot)|^{\beta_w}\big\|_{L^{\gamma}} ds\\
		&\lesssim (1+t)^{-\frac{n}{2}\tron{1-\frac{1}{\gamma}}} \|U\|_{X(T)}^{\alpha_w+\beta_w} \int_0^{t/2}(1+s)^{-\frac{n}{2}\tron{\alpha_w+\beta_w-1}+\alpha_w[\ell_u]^+}ds\\
		&\quad+ (1+t)^{-\frac{n}{2}\tron{\alpha_w+\beta_w-\frac{1}{\gamma}}+\alpha_w[\ell_u]^+} \|U\|_{X(T)}^{\alpha_w+\beta_w}\int_{t/2}^tds.
	\end{align*}
	Using \eqref{eq_lossofdecay} and \eqref{eq_condition_gt_1} again, we may conclude
	\begin{equation}\label{eq_Lk_Thr1}
		\begin{aligned}
			&\|u^{\rm nl}(t,\cdot)\|_{L^{\gamma}}\lesssim (1+t)^{-\frac{n}{2}\tron{1-\frac{1}{\gamma}}+[\ell_u]^+}\|U\|^{p+\alpha}_{X(T)},\\
			&\|v^{\rm nl}(t,\cdot)\|_{L^{\gamma}}\lesssim (1+t)^{-\frac{n}{2}\tron{1-\frac{1}{\gamma}}}\|U\|^{q+\beta}_{X(T)}.
		\end{aligned}
	\end{equation}
	Finally, the terms $\partial_tu$ and $\partial_tv$ can be estimated as follows:
	\begin{align*}
		\|\partial_tw^{\rm nl}(t,\cdot)\|_{L^2} &\lesssim \int_0^{t/2} (1+t-s)^{-\frac{n}{4}-1} \big\||u(s,x)|^{\alpha_w}|v(s,x)|^{\beta_w}\big\|_{L^1\cap L^2} ds\\
		&\qquad+\int_{t/2}^t (1+t-s)^{-1}\big\||u(s,\cdot)|^{\alpha_w}|v(s,\cdot)|^{\beta_w}\big\|_{L^{2}} ds\\
		&\lesssim (1+t)^{-\frac{n}{4}-1} \|U\|_{X(T)}^{\alpha_w+\beta_w} \int_0^{t/2}(1+s)^{-\frac{n}{2}\tron{\alpha_w+\beta_w-1}+\alpha_w[\ell_u]^+}ds\\
		&\quad+ (1+t)^{-\frac{n}{2}\tron{\alpha_w+\beta_w-\frac{1}{2}}+\alpha_w[\ell_u]^+} \|U\|_{X(T)}^{\alpha_w+\beta_w}\int_{t/2}^t (1+t-s)^{-1}ds
	\end{align*}
	for $w\in\{u,v\}$, where we applied Lemma \ref{lemma_G-N_Thr1globalexistence} and Proposition \ref{LinearEstimates} with $(m_1,m_2) =(1,2), (s_1, s_2) = (0,0)$, $j=1$ for $\tau \in [0, t/2)$ and $(m_1,m_2)=(2,2), (s_1, s_2) = (0,0)$, $j=1$ for $\tau \in [t/2, t]$. Thus, \eqref{eq_lossofdecay} and \eqref{eq_condition_gt_1} imply that
	\begin{equation}\label{eq_partial_t_Thr1}
		\begin{aligned}
			&\|\partial_t u^{\rm nl}(t,\cdot)\|_{L^2}\lesssim (1+t)^{-\frac{n}{4}-1+[\ell_u]^+} \|U\|_{X(T)}^{p+\alpha},\\
			&\|\partial_t v^{\rm nl}(t,\cdot)\|_{L^2}\lesssim (1+t)^{-\frac{n}{4}-1}\|U\|_{X(T)}^{q+\beta}.
		\end{aligned}
	\end{equation}
	The estimates \eqref{eq_H1_Thr1}, \eqref{eq_Lk_Thr1} and \eqref{eq_partial_t_Thr1} complete the proof of \eqref{eqNu_Thr1}.
	\noindent By \eqref{keyestimate_Thr1globalexistence}, we may choose $\varepsilon_0>0$ sufficiently small so that, for all initial data satisfying $ \|(U_0,U_1)\|_{\mathcal{D}_{r}}\leq \varepsilon_0$, we have
	\[
	C_0\|(U_0,U_1)\|_{\mathcal{D}_{r}}\leq \varepsilon/3.
	\]
	We may also choose $\varepsilon>0$ sufficiently small so that 
	\[
	C_0\max_{w\in\{u,v\}}\{\varepsilon^{\alpha_w+\beta_w-1}\}\leq 1/3.
	\]
	Hence, $N$ maps $\mathcal{B}_T$ into itself. We next prove the continuity of $N$ in $E_T$. Let $U^{(k)}=(u^k,v^k)$ and $U=(u,v)$ such that $U^{(k)}\to U$ in $E_T$ as $k\to\ity$, with $U^{(k)},U\in\mathcal{B}_T$. We will prove that
	\[
	\sup_{0\leq t\leq T}\|w(t)\|_{\dot{H^1(\Omega_T)}}<\infty
	\]
	for $w\in \{u,v\}$. Indeed, for each $t^*\in [0,T]$, we choose a sequence $t_j\to t^*$ as $j\to \ity$. Then, $w(t_j)$ is uniform bounded in $H^1(\Omega_T)$ so there exists $v\in H^1(\Omega_T)$ such that $w(t_j){\rightharpoonup} v\text{ in }H^1(\Omega_T)$ as $j\to \ity$. Since $H^1(\Omega_T)\Subset L^2(\Omega_T)$, we have $w(t_j)\to v\text{ in }L^2(\Omega_T)$ as $j\to \ity$. On the other hand, the convergence in $\mathcal{C}([0,T];L^2(\Omega_T))$ implies $w(t_j)\to w(t^*)$ in $L^2(\Omega_T)$ as $j\to \ity$. Since the limit is unique, we obtain $w(t^*)=v$ in $L^2(\Omega_T)$. Hence, we can choose $v$ as the presentation of $w(t^*)$ in $H^1(\Omega_T)$. For any $\sigma\in [2,\infty)$ when $n=1,2$ and $\sigma \in [2,6)$ when $n=3$, the Rellich-Kondrachov theorem implies
	\[
	\sup_{0\leq t\leq T}\|w(t)\|_{\dot{L^\sigma(\Omega_T)}}\leq \sup_{0\leq t\leq T}\|w(t)\|_{\dot{H^1(\Omega_T)}}<\infty.
	\]
	From the fact that $\supp(w^{(k)}(t)),\supp(w(t))\subset \Omega_T$, employing Lemma \ref{fractionalGagliardoNirenberg} we derive
	\[
	\begin{aligned}
		\sup_{0\leq t\leq T}\|w^{(k)}(t)-w(t)\|_{L^{\sigma}(\Omega_T)}&\lesssim \tron{\sup_{0\leq t\leq T}\|w^{(k)}(t)-w(t)\|_{L^2(\Omega_T)}}^{1-\vartheta_\sigma}\\
		&\qquad\times \tron{\sup_{0\leq t\leq T}\|w^{(k)}(t)-w(t)\|_{\dot{H}^1(\Omega_T)}}^{\vartheta_\sigma},
	\end{aligned}
	\]
	where $\vartheta_\sigma=n(1/2-1/\sigma)\in [0,1]$. The second term is uniform bound while the first term tends to $0$. We deduce that
	$U^{(k)}(t)\to U(t)$ in $\bigl(L^\sigma(\Omega_T)\bigr)^2$ as $k\to\ity$ for all $t\in [0,T]$. Moreover, using Lemma \ref{fractionalGagliardoNirenberg} again, we have 
	\[
	\begin{aligned}
		\sup_{0\leq t,s\leq T}\|w(t)-w(s)\|_{L^{\sigma}(\Omega_T)}&\lesssim \tron{\sup_{0\leq t\leq T}\|w(t)-w(s)\|_{L^2}(\Omega_T)}^{1-\vartheta_\sigma}\\
		&\qquad \times \tron{\sup_{0\leq t\leq T}\|w(t)-w(s)\|_{\dot{H}^1(\Omega_T)}}^{\vartheta_\sigma}\to 0 
	\end{aligned}
	\]
	when $s\to t$. Hence, we deduce that
	\begin{equation}\label{embed_L_sigma}
		U^{(k)}(t)\to U(t)\quad\text{in}\quad \bigl(\mathcal{C}([0,T];L^\sigma(\Omega_T))\bigr)^2 \text{ as } k\to\ity,
	\end{equation}
	by choosing
	\begin{equation}\label{choice_of_z}
		\sigma:=2\max\{\alpha_w+\beta_w\},\quad w\in\{u,v\}.
	\end{equation}
	The continuity of the corresponding Nemytskii operators combined with H\"older's inequality and \eqref{embed_L_sigma} gives
	\[
	F_w\big(U^{(k)}\big)\to F_w(U) \quad\text{in}\quad L^1([0,T];L^2(\Omega_T))\text{ as } k\to\ity, \qquad w\in\{u,v\},
	\]
	where we have used the estimate
	\[
	\begin{aligned}
		\big||u^k|^{\alpha_w}|v^{(k)}|^{\beta_w}-|u|^{\alpha_w}|v|^{\beta_w}\big|&\lesssim \big||u^k|^{\alpha_w}-|u|^{\alpha_w}\big||v^{(k)}|^{\beta_w}+|u|^{\alpha_w}\big||v^{(k)}|^{\beta_w}-|v|^{\beta_w}\big|.
	\end{aligned}
	\]
	Consequently, one arrives at
	\[
	\|N\big(U^{(k)}\big)-\mathcal{N}[U]\|_{E_T}\lesssim \|F\big(U^{(k)}\big)-F(U)\|_{\bigl(L^1([0,T];L^2(\Omega_T))\bigr)^2}\to 0 \text{ as } k\to\ity.
	\]
	Applying Schauder's fixed point theorem in $E_T$ yields a fixed point $U\in\mathcal{B}_T$ of $N$ on $[0,T]$. Since $U=\mathcal{N}[U]$, the estimate \eqref{keyestimate_Thr1globalexistence} follows that
	\[
	\begin{aligned}
		\|U\|_{X(T)}&\leq C_0\|(U_0,U_1)\|_{\mathcal D_r}+C_0\Bigl(\|U\|_{X(T)}^{p+\alpha-1}+\|U\|_{X(T)}^{q+\beta-1}\Bigr)\|U\|_{X(T)}\\
		&\leq C_0\|(U_0,U_1)\|_{\mathcal D_r}+\frac{2}{3}\|U\|_{X(T)},
	\end{aligned}
	\]
	which is to give immediately $\|U\|_{X(T)}
	\leq 3C_0\|(U_0,U_1)\|_{\mathcal D_r}$. \medskip
	
	\noindent {\bf $\bullet$ Step 4: Construction of a global solution.} We now construct a global solution from the solutions on finite time intervals. For $T_k=k$, where $k\in\mathbb{N}^\ast$, there exists a solution $U^{(k)}=(u^{(k)},v^{(k)})$ to \eqref{Main_system} on $[0,k]$ such that
	\[
	\|U^{(k)}\|_{X(k)}\leq 3C_0\|(U_0,U_1)\|_{\mathcal D_r}=:C,
	\]
	where $C$ is independent of $k$ and $\supp\{u^{(k)}(t),v^{(k)}(t)\}\in \Omega_t$ for $0\leq t\leq k$. Due to the fact
	\[
	\|U^{(k)}\|_{X(1)}\leq \|U^{(k)}\|_{X(k)}\leq C\text{ for $k\geq 1$},
	\]
	it follows that
	$$\{U^{(k)}\}_{k\in \mathbb N} \text{ and }\{\partial_tU^{(k)}\}_{k\in \mathbb N} \text{ are bounded in }(L^{\infty}((0,1),H^1(\Omega_1)))^2 \text{ and }(L^{\infty}((0,1),L^2(\Omega_1)))^2, $$
	respectively. Taking $\sigma$ defined as in \eqref{choice_of_z}, by the Rellich-Kondrachov theorem, we have
	\[
	H^1(\Omega_1)\Subset L^{\sigma}(\Omega_1).
	\]
	By the Aubin--Lions--Simon theorem, there exists a subsequence $\{U^{(k^1_j)}\}_{j\in \mathbb N}$ such that
	\[
	U^{(k^1_{j})}\to U_1\text{ in }\bigl(\mathcal{C}([0,1];L^{\sigma}(\Omega_{1}))\bigr)^2 \text{ as } j\to\ity.
	\]
	Next, let us choose a subsequence $\{U^{(k^2_j)}\}_{j\in \mathbb N}$ of $\{U^{(k^1_j)}\}_{j\in \mathbb N}$ with $k^2_1\geq2$. Then $\{U^{(k^2_j)}\}_{j\in \mathbb N}$ is well defined on $[0,2]$. Repeating the above argument and relabeling the resulting subsequence, we obtain
	\[
	U^{(k^2_{j})}\to U_2\text{ in }\bigl(\mathcal{C}([0,2];L^{\sigma}(\Omega_{2}))\bigr)^2\text{ as } j\to\ity.
	\]
	Since the limit is unique, we have ${\bf 1}_{[0,1]}(t)U_2=U_1$. Proceeding in the same manner, for each $m\in\mathbb{N}^\ast$ we obtain the following chain of subsequences:
	\[
	\left\{U^{(k^1_{j})}\right\}_{j\in \mathbb N}\supset \left\{U^{(k^2_{j})}\right\}_{j\in \mathbb N} \supset\cdots\supset \left\{U^{(k^m_{j})}\right\}_{j\in \mathbb N}\supset \cdots
	\]
	and 
	\[
	U^{(k^m_{j})}\to U_{m} \text{ in }\bigl(\mathcal{C}([0,m];L^{\sigma}(\Omega_{m}))\bigr)^2\text{ as }j\to\ity
	\]
	such that ${\bf1}_{[0,\tilde{m}]}(t)U_m=U_{\widetilde{m}}$ for all $1\leq\widetilde{m}\leq m$. Define the diagonal subsequence by $U^{(n_j)}=U^{k^j_j}$. Then there exists $U=(u,v)$ such that $\supp\{u(t),v(t)\}\subset \Omega_{m}$ for $0\leq t\leq m$ and
	\begin{equation}\label{converge_C_Lq}
		U^{(n_j)}\to U\text{ in }\bigl(\mathcal{C}([0,m];L^{\sigma}(\Omega_m)\bigr)^2\hookrightarrow \bigl(\mathcal{C}([0,m];L^{\sigma}\bigr)^2
	\end{equation}
	as $j\to\ity$ on any finite interval $[0,m]$, where $m\in\mathbb{N}^\ast$. We now show that $U$ is a global (in time) solution to \eqref{Main_system}. For each $m\in\mathbb N^\ast$, we have 
	\[
	\sup_{j\in \mathbb N:\,n_j\geq m}\|U^{(n_j)}\|_{X(m)}\leq \varepsilon.
	\]
	Therefore, one finds
	\begin{align*}
		U^{(n_j)}\stackrel{*}{\rightharpoonup} U\text{ in }(L^\infty([0,m],H^1))^2
	\end{align*}
	as $j\to\ity$ and there exists $V$ such that $\partial_tU^{(n_j)}\stackrel{*}{\rightharpoonup} V\text{ in }(L^\infty([0,m],L^2))^2$ as $j\to\ity$. Since $\sigma\geq2$ and $\Omega_m$ has finite measure, from \eqref{converge_C_Lq} it implies $U^{(n_j)}\to U$ in $\bigl(\mathcal{C}([0,m];L^2)\bigr)^2$ as $j\to\ity$. Hence, $V=\partial_tU$ in the weak sense. Moreover, for each $w\in\{u,v\}$ it holds $2(\alpha_w+\beta_w)\leq \sigma$. Therefore, again using that $\Omega_m$ has finite measure, \eqref{converge_C_Lq} yields $w^{(n_j)}\to w$ in $\mathcal{C}\big([0,m];L^{2(\alpha_w+\beta_w)}\big)$ as $j\to \ity$. Consequently, one arrives at
	\begin{equation}\label{integral_in_L1}
		\big|u^{(n_j)}\big|^{\alpha_w}\big|v^{(n_j)}\big|^{\beta_w}\to |u|^{\alpha_w}|v|^{\beta_w}\text{ in }L^1([0,m];L^2)\text{ as }j\to\ity, \text{ for }w\in\{u,v\}.
	\end{equation}
	For each $n_j$ and $w\in\{u,v\}$, we have
	\[
	w^{(n_j)}(t,x)=w^{\rm lin}(t,x)+\int_{0}^{t}\mathcal{K}(t-s,x)\ast_x\bigl(|u^{(n_j)}(s,x)|^{\alpha_w}|v^{(n_j)}(s,x)|^{\beta_w}\bigr)\,ds,\qquad t\in[0,m].
	\]
	Define
	\[
	\widetilde w(t,x):=w^{\rm lin}(t,x)+\int_{0}^{t}\mathcal{K}(t-s,x)\ast_x\bigl(|u(s,x)|^{\alpha_w}|v(s,x)|^{\beta_w}\bigr)\,ds
	\]
	and set $z_w^{(n_j)}:=w^{(n_j)}-\widetilde w$. Then, $z_w^{(n_j)}$ is the weak solution to
	\[
	\left\{
	\begin{aligned}
		&\partial_t^2z_w^{(n_j)}-\Delta z_w^{(n_j)}+\partial_tz_w^{(n_j)}=F_w(U^{(n_j)})-F_w(U),&&x\in\mathbb R^n,\ t>0,\\
		&z_w^{(n_j)}(0,x)=\partial_tz_w^{(n_j)}(0,x)=0,&&x\in\mathbb R^n,
	\end{aligned}
	\right.
	\]
	whose source term belongs to $L^1([0,m];L^2)\hookrightarrow
	L^1([0,m];H^{-1})$. The application of Lemma \ref{regularity_K} yields
	\[
	\sup_{0\leq t\leq m} \Big(\|w^{(n_j)}(t)-\widetilde w(t)\|_{H^1}+\|\partial_tw^{(n_j)}(t)-\partial_t\widetilde w(t)\|_{L^2}\Big)\lesssim \|F_w(U^{(n_j)})-F_w(U)\|_{L^1([0,m];L^2)}\to 0
	\]
	as $j\to\infty$, by \eqref{integral_in_L1}. Thus, it follows that $w^{(n_j)}\to \widetilde w$ and $\partial_tw^{(n_j)}\to \partial_t\widetilde w$ in $L^\infty([0,m];H^1)$ and $L^\infty([0,m];L^2)$, respectively, as $j\to \ity$. By Lemma \ref{regularity_K} and \eqref{integral_in_L1}, we have 
	\[
	\widetilde w\in \mathcal{C}([0,m];H^1)\cap \mathcal{C}^1([0,m];L^2).
	\]
	Since $H^1\hookrightarrow L^2$, it holds $w^{(n_j)}\to \widetilde w$ in $L^\infty([0,m];L^2)$ as $j\to \ity$. Since $w^{(n_j)}-\widetilde w\in \mathcal{C}([0,m];L^2)$, the essential supremum above coincides with the usual supremum, i.e. $w^{(n_j)}\to \widetilde w$ in $\mathcal{C}([0,m];L^2)$ as $j\to \ity$. On the other hand, one sees $w^{(n_j)}\to w$ in $\mathcal{C}([0,m];L^2)$ as $j\to \ity$. By uniqueness of the limit, we get $w=\widetilde w$ in $\mathcal{C}([0,m];L^2)$. By identifying $w$ with the representative $\widetilde w$, we gain
	\[
	U\in \left(\mathcal{C}([0,m];H^1)\cap \mathcal{C}^1([0,m];L^2)\right)^2
	\]
	and $U$ satisfies \eqref{duhamel_formula}. Since $m$ is arbitrary, one finds that
	\[
	U\in \bigl(\mathcal{C}([0,\infty);H^1)\cap \mathcal{C}^1([0,\infty);L^2)\bigr)^2
	\]
	is a global (in time) mild solution of \eqref{Main_system}. To verify $U\in(\mathcal{C}([0,\infty);L^r))^2$, we notice that
	\[
	\|w^{(n_j)}\|_{L^r}\lesssim \|U^{(n_j)}\|_{X(m)}<\infty
	\]
	for $w\in \{u,v\}$ and for arbitrary $m\geq 0$. Passing to a subsequence if needed, we have
	\[
	U^{(n_j)}=(u^{(n_j)},v^{(n_j)})\stackrel{*}{\rightharpoonup} V=(\overline{u},\overline{v})\text{ in }(L^\infty([0,m];L^r))^2\text{ as }j\to\ity.
	\]
	Thus for any $\Phi(t,x)\in \mathcal{C}^\infty_{\rm c}([0,m]\times \mathbb{R}^n)$, it holds
	\[
	\int_{0}^m\int_{\mathbb{R}^n}w^{(n_j)}\Phi(t,x)dxdt\to \int_{0}^m\int_{\mathbb{R}^n}\overline{w}\Phi(t,x)dxdt\text{ as }j\to\ity.
	\]
	The strong convergent in $L^\infty([0,m],L^2)$ implies that 
	\[
	\int_{0}^m\int_{\mathbb{R}^n}w^{(n_j)}\Phi(t,x)dxdt\to \int_{0}^m\int_{\mathbb{R}^n}w\Phi(t,x)dxdt\text{ as }j\to \ity.
	\]
	Since the limit is unique and $w,\overline{w}\in L^1_{\rm loc}([0,m]\times \mathbb{R}^n)$, by the fundamental lemma of calculus of variations, we obtain $w=\overline{w}$ a.e. in $[0,m]\times \mathbb{R}^n$ for $w\in \{u,v\}$. Hence, we get
	\[
	U^{(n_j)}\stackrel{*}{\rightharpoonup} U\text{ in }(L^\infty([0,m];L^r))^2\text{ as }j\to\ity
	\]
	for arbitrary $m\geq 0$ and $\supp (w(t,\cdot)-w(s,\cdot))\subset \Omega_{\max\{t,s\}}$ for $w\in\{u,v\}$ and $0\leq t,s\leq m$. For $1<r\leq 2$, we get
	\[
	\|w(t)-w(s)\|_{L^r}\leq |\Omega_{\max\{t,s\}}|^{1/r-1/2}\|w(t)-w(s)\|_{L^2}
	\]
	for $t,s\in[0,m]$, which is to conclude the desired result by letting $s\to t$. Finally, for every $m\in \mathbb N^*$, the weak-$*$ lower semicontinuity of the weighted $L^\infty$ norm gives
	\[
	\|U\|_{X(m)}\leq \liminf_{j\to\ity}\|U^{(n_j)}\|_{X(m)}\leq 3C_0\|(U_0,U_1)\|_{\mathcal{D}_r}.
	\]
	Since the right hand side is independent of $m$, we obtain $\|U\|_{X(\infty)}\leq 3C_0\|(U_0,U_1)\|_{\mathcal{D}_r}$. Thus, $U$ satisfies the decay estimates in Theorem \ref{Thr_1_Globalexistence}.
	\subsection{Proof of Theorem \ref{Thr_1_Globalexistence} without the compact-support assumption}
	Let us divide the proof into some steps as follows.\medskip
	
	\noindent
	\textbf{Step 1: Approximation of the initial data and uniform bounds.}
	Assume that $(U_0,U_1)\in\mathcal D_r$ and $ \|(U_0,U_1)\|_{\mathcal D_r}\leq\varepsilon_0/2$. Since $\mathcal{C}^\infty_{\rm c}$ is dense in all the spaces appearing in $\mathcal D_r$, there exists a sequence
	\[
	(U_0^k,U_1^k)\in\bigl(\mathcal{C}^\infty_{\rm c}\bigr)^2\times\bigl(\mathcal{C}^\infty_{\rm c}\bigr)^2
	\]
	such that
	\begin{equation}\label{converge_initialdata}
		(U_0^k,U_1^k)\to (U_0,U_1)\qquad\text{in }\mathcal D_r \text{ as }k\to \ity.
	\end{equation}
	Without loss of generality, we may assume that $\|(U_0^k,U_1^k)\|_{\mathcal D_r}\leq\varepsilon_{0}$ for any $k\in\mathbb N$. By the compact-support version of the global existence theorem, for each $k\in\mathbb N$, there exists a global mild solution $U^{(k)}=(u^k,v^{(k)})$. Moreover, we get
	\begin{equation}\label{remove_compact_eq0}
		\|U^{(k)}\|_{X(\infty)}\leq 3C_0\|(U_0^k,U_1^k)\|_{\mathcal D_r}\leq C\varepsilon_{0},
	\end{equation}
	where the constant $C$ is independent of $k$ and of the radius of the support of $(U_0^k,U_1^k)$. Fix $T>0$. It follows from \eqref{remove_compact_eq0} that $\{U^{(k)}\}_{k\in\mathbb N}$, $\{\partial_tU^{(k)}\}_{k\in\mathbb N}$ are bounded in
	$$\bigl(L^\infty([0,T];H^1)\bigr)^2 \cap \bigl(L^\infty([0,T];L^r)\bigr)^2 \text{ and }\bigl(L^\infty([0,T];L^2)\bigr)^2, $$
	respectively. After passing to a subsequence, there exist $U^{H}=(u^H,v^H)$, $U^{r}=(u^r,v^r)$ and $V=(v^1,v^2)$ such that
	\begin{align}
		&U^{(k)}\stackrel{*}{\rightharpoonup} U^{(H)}\quad\text{in }\bigl(L^\infty([0,T];H^1)\bigr)^2 ,\label{remove_compact_eq1}\\
		&U^{(k)}\stackrel{*}{\rightharpoonup}U^{(r)}\quad\text{in }\bigl(L^\infty([0,T];L^r)\bigr)^2 ,\label{remove_compact_eq2}\\
		&\partial_tU^{(k)}\stackrel{*}{\rightharpoonup} V\quad\text{in }\bigl(L^\infty([0,T];L^2)\bigr)^2, \label{remove_compact_eq3}
	\end{align}
	as $k\to \ity$. Let us now identify the first two weak-$*$ limits. For any $\Psi\in \mathcal{C}^\infty_{\rm c}((0,T)\times\mathbb R^n)$, since $\Psi$ belongs to both
	\[
	\bigl(L^1([0,T];H^{-1})\bigr)^2 \quad\text{and}\quad \bigl(L^1([0,T];L^{r'})\bigr)^2\quad\text{ with } r'=\frac{r}{r-1},
	\]
	the convergence in \eqref{remove_compact_eq1} and
	\eqref{remove_compact_eq2} implies
	\[
	\int_0^T\int_{\mathbb R^n}w^{(k)}\cdot\Psi\,dx\,dt
	\to
	\int_0^T\int_{\mathbb R^n}w^{(H)}\cdot\Psi\,dx\,dt 
	\]
	and
	\[
	\int_0^T\int_{\mathbb R^n}w^{(k)}\cdot\Psi\,dx\,dt
	\to
	\int_0^T\int_{\mathbb R^n}w^{(r)}\cdot\Psi\,dx\,dt
	\]
	as $k\to \ity$ for $w\in \{u,v\}$. Thus, it follows
	\[
	U^{(H)}=U^{(r)}
	\quad\text{in }
	\bigl(\mathcal D'((0,T)\times\mathbb R^n)\bigr)^2,
	\]
	and therefore, it is still valid a.e on $(0,T)\times\mathbb R^n$. We denote this common limit by $U$. Similarly, we can prove that $V=\partial_tU$ in the weak sense. \medskip
	
	\noindent \textbf{Step 2: Compactness and convergence of the nonlinear terms.} Let $M>0$. The restrictions of $\{U^{(k)}\}_{k\in \mathbb N}$ to $B_M:=\{x\in \mathbb R^n:|x|\leq M\}$ are uniformly bounded in
	\[
	\bigl(L^\infty([0,T];H^1(B_M))\bigr)^2,
	\]
	whereas $\{\partial_tU^{(k)}\}_{k\in \mathbb N}$ is uniformly bounded in
	\[
	\bigl(L^\infty([0,T];L^2(B_M))\bigr)^2.
	\]
	By the Rellich-Kondrachov theorem, it holds $H^1(B_M)\Subset L^{\sigma}(B_M)$. The Aubin-Lions-Simon compactness theorem implies that, after passing to a subsequence, there exists
	\[
	U_M\in\bigl(\mathcal{C}([0,T];L^{\sigma}(B_M))\bigr)^2
	\]
	such that
	\[
	U^{(k)}\to U_M
	\quad\text{in }\bigl(\mathcal{C}([0,T];L^{\sigma}(B_M))\bigr)^2.
	\]
	We now identify this local strong limit with the restriction of the weak-$*$ limit obtained in Step 1. Indeed, the relation \eqref{remove_compact_eq1} implies
	\[
	U^{(k)}\rightharpoonup {\bf 1}_{B_M}(x)U
	\quad\text{in }\bigl(L^2([0,T];L^2(B_M))\bigr)^2\text{ as }k\to \ity.
	\]
	To establish this, taking any $\Psi\in\bigl(L^2([0,T];L^2(B_M))\bigr)^2$ we may extend it by zero outside $B_M$. Since $T<\infty$ and $L^2\hookrightarrow H^{-1}$, the extension belongs to 
	\[
	\bigl(L^1([0,T];H^{-1})\bigr)^2,
	\]
	so that the preceding weak convergence follows directly from \eqref{remove_compact_eq1}. On the other hand, the strong convergence to $U_M$ implies weak convergence to $U_M$ in the same space. By uniqueness of the weak limit, one has
	\[
	U_M={\bf 1}_{B_M}(x)U\quad\text{in }\bigl(L^2([0,T];L^2(B_M))\bigr)^2.
	\]
	As a result, we choose $U_M$ as the continuous representative of ${\bf 1}_{B_M}(x)U$. Moreover, if $0<M_1<M_2$, then it holds
	\[
	U_{M_1}={\bf 1}_{B_{M_1}}(x)U_{M_2}
	\quad\text{in }\bigl(L^2([0,T];L^2(B_{M_1}))\bigr)^2.
	\]
	Since both sides belong to 
	\[
	\bigl(\mathcal{C}([0,T];L^2(B_{M_1}))\bigr)^2,
	\]
	the equality is valid for any $t\in[0,T]$. Next, let us apply a diagonal argument with respect to $T,M\in\mathbb N^*$. We first have the following chain of subsequences:
	\[
	\left\{U^{k^1_{j}}\right\}_{j\in \mathbb N}\supset \left\{U^{k^2_{j}}\right\}_{j\in \mathbb N} \supset\cdots\supset \left\{U^{k^m_{j}}\right\}_{j\in \mathbb N}\supset \cdots
	\]
	where
	\[
	U^{k^i_j}\to U \text{ in }(\mathcal{C}[0,i];L^\sigma(B_i))^2\text{ as }j\to \ity,\quad\forall i\in \mathbb N^*.
	\]
	Using the same compatibility argument on overlapping time intervals and spatial balls, we may choose the diagonal subsequence, still denoted by $\{U^{(k)}\}_{j\in \mathbb N}$, and a single locally continuous representative of $U$ such that
	\begin{equation}\label{remove_compact_eq4}
		U^{(k)}\to U
		\quad\text{in }\bigl(\mathcal{C}([0,T];L^{\sigma}(B_M))\bigr)^2\text{ as }k\to \ity
	\end{equation}
	for any $T,M>0$. Now, we identify the weak limit of the time derivatives. Let $\varphi\in \mathcal{C}^\infty_{\rm c}(0,T)$ and $\zeta\in \mathcal{C}^\infty_{\rm c}$. For any $k\in\mathbb N$, one sees
	\[
	\int_0^T\langle U^{(k)}(t),\zeta\rangle\varphi'(t)\,dt=-\int_0^T\langle\partial_tU^{(k)}(t),\zeta\rangle\varphi(t)\,dt.
	\]
	Passing to the limit using \eqref{remove_compact_eq3} and \eqref{remove_compact_eq4}, we obtain
	\[
	\int_0^T\langle U(t),\zeta\rangle\varphi'(t)\,dt=-\int_0^T\langle V(t),\zeta\rangle\varphi(t)\,dt,
	\]
	i.e. $V=\partial_tU$ in the sense of distributions and
	\[
	\partial_tU\in \bigl(L^\infty([0,T];L^2)\bigr)^2.
	\]
	Since $2(\alpha_w+\beta_w)\leq \sigma$ with $w\in\{u,v\}$, linking of the convergence \eqref{remove_compact_eq4}, the finite measure of $B_M$ and the continuity of the corresponding Nemytskii operators we deduce
	\begin{equation}\label{remove_compact_eq6}
		F_w\big(U^{(k)}\big)\to  F_w(U)\quad\text{in }L^1([0,T];L^2(B_M))\text{ as }k\to \ity,\qquad w\in\{u,v\}.
	\end{equation}
	Moreover, by the Sobolev embedding we have
	\[
	\begin{aligned}
		\|F_w(U(t))\|_{L^2}\leq\|u(t)\|_{L^{2(\alpha_w+\beta_w)}}^{\alpha_w}\|v(t)\|_{L^{2(\alpha_w+\beta_w)}}^{\beta_w}\lesssim\|u(t)\|_{H^1}^{\alpha_w}\|v(t)\|_{H^1}^{\beta_w},
	\end{aligned}
	\]
	where we use $\alpha_w+\beta_w<3$ when $n=3$. It follows that
	\[
	F_w(U)\in L^\infty([0,T];L^2)
	\subset L^1([0,T];L^2).
	\]
	
	\medskip
	\noindent
	\textbf{Step 3: Passing to the limit and recovering the mild formulation.}
	For any $k\in\mathbb N$ and $w\in\{u,v\}$, the representation formula of mild solutions gives
	\[
	\begin{aligned}
		w^{(k)}(t)
		&=\bigl(\mathcal K(t)+\partial_t\mathcal K(t)\bigr)*w_0^k
		+\mathcal K(t)*w_1^k
		+\int_0^t\mathcal K(t-s)*F_w\big(U^{(k)}(s)\big)\,ds.
	\end{aligned}
	\]
	The linear part is the energy solution, and hence a distributional
	solution, to the homogeneous damped wave equation with initial data
	$(w_0^k,w_1^k)$. Since $U^{(k)}$ is a mild solution, it holds
	\[
	F_w\big(U^{(k)}\big)\in L^1([0,T];L^2).
	\]
	After applying Lemma \ref{regularity_K} with $g=F_w\big(U^{(k)}\big)$, we conclude that the Duhamel term is a distributional solution of the inhomogeneous equation with source $F_w\big(U^{(k)}\big)$ and zero initial data. Consequently, $w^{(k)}$ is a distributional solution to
	\[
	\partial_t^2w^{(k)}-\Delta w^{(k)}+\partial_tw^{(k)}=F_w\big(U^{(k)}\big)
	\]
	with initial data $(w_0^k,w_1^k)$. For any $\Phi\in \mathcal{C}^\infty_{\rm c}([0,T)\times\mathbb R^n)$, we achieve
	\begin{equation}\label{remove_compact_eq7}
		\begin{aligned}
			\int_0^T\int_{\mathbb R^n}\Bigl(-\partial_tw^{(k)}\,\partial_t\Phi+\nabla w^{(k)}\cdot\nabla\Phi+\partial_tw^{(k)}\,\Phi-F_w\big(U^{(k)}\big)\Phi\Bigr)\,dx\,dt=\int_{\mathbb R^n}w_1^k(x)\Phi(0,x)\,dx.
		\end{aligned}
	\end{equation}
	Since $\Phi$ has compact support, there exists $M_\Phi>0$ such that $\operatorname{supp}\Phi\subset [0,T)\times B_{M_\Phi}$. Fix $M\geq M_\Phi$, by \eqref{remove_compact_eq6}, one has
	\[
	\left|\int_0^T\int_{\mathbb R^n}\bigl(F_w\big(U^{(k)}\big)-F_w(U)\bigr)\Phi\,dx\,dt\right|\leq \|F_w\big(U^{(k)}\big)-F_w(U)\|_{L^1([0,T];L^2(B_M))}\|\Phi\|_{L^\infty([0,T];L^2(B_M))}\to 0
	\]
	as $k\to\ity$. Furthermore, the weak-$*$ convergence in \eqref{remove_compact_eq1} and \eqref{remove_compact_eq3} together with the identity $V=\partial_tU$ in Step 2 implies
	\begin{align*}
		\int_0^T\int_{\mathbb R^n}\nabla w^{(k)} \cdot\nabla\Phi\,dx\,dt &\to \int_0^T\int_{\mathbb R^n}\nabla w\cdot\nabla\Phi\,dx\,dt, \\
		\int_0^T\int_{\mathbb R^n}\partial_tw^{(k)}\,\partial_t\Phi\,dx\,dt &\to \int_0^T\int_{\mathbb R^n}\partial_tw\,\partial_t\Phi\,dx\,dt, \\
		\int_0^T\int_{\mathbb R^n}\partial_tw^{(k)}\,\Phi\,dx\,dt &\to \int_0^T\int_{\mathbb R^n} \partial_tw\,\Phi\,dx\,dt,
	\end{align*}
	as $k\to \ity$. The convergence of the initial data also gives
	\[
	\int_{\mathbb R^n}w_1^{(k)}(x)\Phi(0,x)\,dx\to \int_{\mathbb R^n}w_1(x)\Phi(0,x)\,dx\text{ as }k\to \ity.
	\]
	Passing to the limit in \eqref{remove_compact_eq7}, we obtain
	\[
	\begin{aligned}
		\int_0^T\int_{\mathbb R^n}\Bigl(-\partial_tw\,\partial_t\Phi+\nabla w\cdot\nabla\Phi+\partial_tw\,\Phi-F_w(U)\Phi\Bigr)\,dx\,dt=\int_{\mathbb R^n}w_1(x)\Phi(0,x)\,dx.
	\end{aligned}
	\]
	We next identify the initial displacement. For any fixed $M>0$, the relation \eqref{remove_compact_eq4} with $\sigma\geq2$ entails $w^{(k)}(0)\to  w(0)$ in $L^2(B_M)$ as $k\to \ity$. On the other hand, by \eqref{converge_initialdata} it holds $w^{(k)}(0)=w_0^{(k)}\to  w_0$ in $L^2(B_M)$ as $k\to \ity$. By uniqueness of the limit, $w(0)=w_0$ in $L^2(B_M)$. Since $M>0$ is arbitrary and both functions belong to
	$L^2$, it follows that $w(0)=w_0$ in $L^2$. The standard weak-continuity theorem for energy solutions implies that $t\mapsto w(t)$ is weakly continuous in $H^1$ and $t\mapsto \partial_tw(t)$ is weakly continuous in $L^2$. Consequently, the traces $w(0)\in H^1$ and $\partial_tw(0)\in L^2$ are well defined. The weak $H^1$-trace of $w$ at $t=0$ has the same image in $L^2$ as this strong $L^2$-trace. Consequently, one claims that
	$$w(0)=w_0 \quad\text{in } H^1. $$
	Next, let identify the initial velocity. Since $w\in L^\infty(0,T;H^1)$, $\partial_tw\in L^\infty(0,T;L^2)$ and $F_w(U)\in L^1(0,T;L^2)$, the distributional equation yields
	\[
	\partial_t^2w=\Delta w-\partial_tw+F_w(U)\in L^1(0,T;H^{-1}),
	\]
	i.e.
	\[
	\partial_tw\in
	W^{1,1}(0,T;H^{-1})
	\cap L^\infty(0,T;L^2).
	\]
	Using the distributional equation and integrating by parts in time, we obtain
	\[
	\begin{aligned}
		\int_0^T\int_{\mathbb R^n}\Bigl(-\partial_tw\,\partial_t\Phi+\nabla w\cdot\nabla\Phi+\partial_tw\,\Phi-F_w(U)\Phi\Bigr)\,dx\,dt=\int_{\mathbb R^n}\partial_tw(0,x)\Phi(0,x)\,dx
	\end{aligned}
	\]
	for any $\Phi\in \mathcal{C}^\infty_{\rm c}([0,T)\times\mathbb R^n)$. Comparing this identity with the weak formulation obtained above, we find
	\[
	\int_{\mathbb R^n}\bigl(\partial_tw(0,x)-w_1(x)\bigr)\Phi(0,x)\,dx=0.
	\]
	Taking $\Phi(t,x)=\eta(t)\psi(x)$, where $\eta\in \mathcal{C}^\infty_{\rm c}([0,T)),\eta(0)=1,\psi\in \mathcal{C}^\infty_{\rm c}$, we conclude that
	\[
	\int_{\mathbb R^n}\bigl(\partial_tw(0,x)-w_1(x)\bigr)\psi(x)\,dx=0
	\]
	for any $\psi\in \mathcal{C}^\infty_{\rm c}$. Hence, we get
	\[
	\partial_tw(0)=w_1
	\quad\text{in }L^2.
	\]
	Therefore,
	\[
	w\in L^\infty([0,T];H^1)\cap W^{1,\infty}([0,T];L^2)
	\]
	is a distributional solution, belonging to the energy class, to the linear damped wave equation with source $F_w(U)$ and initial data $(w_0,w_1)$. \medskip
	
	Let us now recover the mild formulation and the strong energy regularity.
	For each $w\in\{u,v\}$, define
	\[
	\begin{aligned}
		\widetilde w(t)
		&:=\bigl(\mathcal K(t)+\partial_t\mathcal K(t)\bigr)*w_0
		+\mathcal K(t)*w_1
		+\int_0^t\mathcal K(t-s)*F_w(U(s))\,ds.
	\end{aligned}
	\]
	By Step 2, it holds $F_w(U)\in L^1([0,T];L^2)$. The homogeneous linear part belongs to
	\[
	\mathcal{C}([0,T];H^1)
	\cap \mathcal{C}^1([0,T];L^2)
	\]
	and is a distributional solution to the homogeneous damped wave equation with initial data $(w_0,w_1)$. Applying Lemma \ref{regularity_K} with $g=F_w(U)$ to the Duhamel term, we obtain
	\[
	\widetilde w\in \mathcal{C}([0,T];H^1)
	\cap \mathcal{C}^1([0,T];L^2).
	\]
	Moreover, $\widetilde w$ is a distributional solution to $\partial_t^2\widetilde w-\Delta\widetilde w
	+\partial_t\widetilde w=F_w(U)$ with initial data $\widetilde w(0)=w_0$, $\partial_t\widetilde w(0)=w_1$. On the other hand, the first part of this step shows that $w$ is a distributional solution of the same equation with the same initial data. Setting $h_w:=w-\widetilde w$ we observe that
	\[
	h_w\in L^\infty([0,T];H^1)\cap W^{1,\infty}([0,T];L^2)
	\]
	and $h_w$ satisfies $\partial_t^2h_w-\Delta h_w+\partial_th_w=0$ in the sense of distributions. In particular, one recognizes
	\[
	\partial_t^2h_w=\Delta h_w-\partial_th_w
	\in L^\infty(0,T;H^{-1}).
	\]
	and 
	\[
	h_w(0)=w(0)-\widetilde w(0)=0
	\;\text{ in }H^1,\quad \partial_th_w(0)=\partial_tw(0)-\partial_t\widetilde w(0)=0\;\text{ in }L^2.
	\]
	We may now apply the standard time-regularization argument by multiplying both sides of the equation by $\partial_th_w$. More precisely, we obtain
	\[
	\begin{aligned}
		&\frac12\|\partial_th_w(t)\|_{L^2}^2+\frac12\|\nabla h_w(t)\|_{L^2}^2+\int_0^t\|\partial_th_w(s)\|_{L^2}^2\,ds=\frac12\Big(\|\partial_th_w(0)\|_{L^2}^2+ \|\nabla h_w(0)\|_{L^2}^2\Big)=0
	\end{aligned}
	\]
	for any $t\in[0,T]$. Therefore, $\partial_th_w=0$ and $\nabla h_w=0$ a.e on $[0,T]\times\mathbb R^n$. Since 
	\[
	h_w\in W^{1,\infty}([0,T];L^2)
	\]
	and $h_w(0)=0$, we have
	\[
	h_w(t)=h_w(0)+\int_0^t\partial_th_w(s)\,ds=0 \quad\text{in }L^2
	\]
	for any $t\in[0,T]$. Consequently, it holds
	\[
	w=\widetilde w \quad\text{in }L^\infty([0,T];H^1)\cap W^{1,\infty}([0,T];L^2).
	\]
	Thus, $\widetilde w$ is a continuous energy representative of $w$. Identifying $w$ with this representative, we obtain
	\[
	w\in \mathcal{C}([0,T];H^1) \cap \mathcal{C}^1([0,T];L^2).
	\]
	From the presentation formula
	\[
	\begin{aligned}
		w(t)&=\bigl(\mathcal K(t)+\partial_t\mathcal K(t)\bigr)*w_0+\mathcal K(t)*w_1+\int_0^t\mathcal K(t-s)*F_w(U(s))\,ds,\qquad w\in\{u,v\},
	\end{aligned}
	\]
	one has
	\begin{equation}\label{remove_compact_eq10}
		U\in\left(\mathcal{C}([0,T];H^1)\cap \mathcal{C}^1([0,T];L^2)\right)^2
	\end{equation}
	and $U$ satisfies the mild formulation on $[0,T]$.
	
	\noindent
	\textbf{Step 4. Continuity in $L^r$ and passage of the decay estimates.}
	Set
	\[
	Y_r:=L^1\cap L^r,\qquad\|f\|_{Y_r}:=\|f\|_{L^1}+\|f\|_{L^r}.
	\]
	By the same Gagliardo--Nirenberg estimates used in the proof of the nonlinear estimates, with $\gamma=1$ and $\gamma=r$, we have $F_w(U)\in L^\infty([0,T];Y_r)\subset L^1([0,T];Y_r)$ for each $w\in\{u,v\}$. Let $S_1(t)f:=\mathcal{K}(t)*f$. The linear estimates in Proposition \ref{LinearEstimates} imply that
	\begin{equation}\label{remove_compact_eq8}
		\sup_{0\leq \theta\leq T}\|S_1(\theta)\|_{\mathcal L(Y_r,L^r)}\leq C_T.
	\end{equation}
	Moreover, for any $f\in Y_r$ one has
	\begin{equation}\label{remove_compact_eq9}
		S_1(\theta)f\to S_1(\theta_0)f\quad\text{in }L^r\quad\text{as }\theta\to \theta_0.
	\end{equation}
	Indeed, let first $f\in \mathcal{C}^\infty_{\rm c}$. Applying Proposition \ref{LinearEstimates} with $\rho=q=r,s_1=s_2=0,j=1$, we obtain
	\[
	\sup_{0\leq \theta\leq T}\|\partial_\theta S_1(\theta)f\|_{L^r}
	\le C_{T,f}
	\]
	for any fixed $T>0$. Therefore, for $\theta,\theta_0\in[0,T]$ one claims that
	\[
	\|S_1(\theta)f-S_1(\theta_0)f\|_{L^r}\leq C_{T,f}|\theta-\theta_0|,
	\]
	which proves \eqref{remove_compact_eq9} for smooth $f$. For a general $f\in Y_r$, let us choose $f_m\in \mathcal{C}^\infty_{\rm c}$ such that $f_m\to f$ in $Y_r$. By \eqref{remove_compact_eq8}, one obtains
	\[
	\begin{aligned}
		\|S_1(\theta)f-S_1(\theta_0)f\|_{L^r}\leq 2C_T\|f-f_m\|_{Y_r}+\|S_1(\theta)f_m-S_1(\theta_0)f_m\|_{L^r}.
	\end{aligned}
	\]
	Letting first $\theta\to \theta_0$ and then $m\to\infty$, we arrive at \eqref{remove_compact_eq9}. For $w\in\{u,v\}$, considering
	\[
	w^{\mathrm{nl}}(t):=\int_0^tS_1(t-s)F_w(U(s))\,ds
	\]
	we write
	\[
	\begin{aligned}
		w^{\mathrm{nl}}(t)-w^{\mathrm{nl}}(\tau)&=\int_0^\tau \bigl(S_1(t-s)-S_1(\tau-s)\bigr)F_w(U(s))\,ds+\int_\tau^t S_1(t-s)F_w(U(s))\,ds\\
		&=:I_1(t,\tau)+I_2(t,\tau)
	\end{aligned}
	\]
	with $t>\tau$. By \eqref{remove_compact_eq8}, one derives
	\[
	\|I_2(t,\tau)\|_{L^r}\leq C_T\int_\tau^t \|F_w(U(s))\|_{Y_r}\,ds\to 0
	\]
	as $t\to\tau^+$. For a.e $s\in(0,\tau)$, the relation \eqref{remove_compact_eq9} yields
	\[
	\bigl(S_1(t-s)-S_1(\tau-s)\bigr)F_w(U(s))\to 0 \quad\text{in }L^r
	\]
	and
	\[
	\begin{aligned}
		\left\|\bigl(S_1(t-s)-S_1(\tau-s)\bigr)F_w(U(s))\right\|_{L^r}\leq 2C_T\|F_w(U(s))\|_{Y_r},
	\end{aligned}
	\]
	since the function on the right-hand side belongs to $L^1(0,T)$. Therefore, the dominated convergence theorem for Bochner integrals therefore gives $\|I_1(t,\tau)\|_{L^r}\to 0$. The case $t<\tau$ is treated similarly. At $\tau=0$, we have
	\[
	\|w^{\mathrm{nl}}(t)\|_{L^r}\leq C_T\int_0^t\|F_w(U(s))\|_{Y_r}\,ds\to 0
	\]
	as $t\to 0^{+}$, i.e. $w^{\mathrm{nl}}
	\in \mathcal{C}([0,T];L^r)$. The linear part also belongs to $\mathcal{C}([0,T];L^r)$ by the linear estimates. Consequently, we conclude that
	\begin{equation}\label{remove_compact_eq11}
		w\in \mathcal{C}([0,T];L^r).
	\end{equation}
	Combining \eqref{remove_compact_eq10} and \eqref{remove_compact_eq11}, we show that
	\[
	U\in\left(\mathcal{C}([0,T];H^1\cap L^r)\cap \mathcal{C}^1([0,T];L^2)\right)^2.
	\]
	It remains to pass the weighted decay estimates to the limit. The solutions $U^{(k)}$ satisfy the same weighted estimates, uniformly with respect to $k$. The fact is that multiplication with the bounded weight $(1+t)^a$ preserves weak-$*$ convergence on $[0,T]$, so the weak-$*$ lower semicontinuity of the $L^\infty([0,T];L^2)$ norm gives
	\[
	\operatorname*{ess\,sup}_{0\leq t\leq T}(1+t)^a\|w(t)\|_{L^2} \leq \liminf_{k\to\infty} \operatorname*{ess\,sup}_{0\leq t\leq T}(1+t)^a\|w^{(k)}(t)\|_{L^2}
	\]
	for each $w\in\{u,v\}$. The same argument applies to the $L^r$, $\dot H^1$, and time-derivative terms occurring in the definition of the $X(T)$ norm. Therefore, one may claim
	\[
	\begin{aligned}
		\|U\|_{X(T)}\leq \liminf_{k\to\infty}\|U^{(k)}\|_{X(T)}\leq 3C_0\lim_{k\to\infty}\|(U_0^k,U_1^k)\|_{\mathcal D_r}=3C_0\|(U_0,U_1)\|_{\mathcal D_r},
	\end{aligned}
	\]
	where $C_0$ is independent of $T$. Since $U$ is continuous in all the spaces involved in the definition of $X(T)$, the essential suprema in the preceding estimates coincide with the usual suprema. Since $T>0$ is arbitrary, we obtain
	\[
	\|U\|_{X(\infty)}\leq 3C_0\|(U_0,U_1)\|_{\mathcal D_r}
	\]
	and $U$ is a global (in time) mild solution to \eqref{Main_system} satisfying all the stated decay estimates. In this way, our proof is established.

	\section{Nonexistence of global mild solutions}\label{sec3}
	\subsection{Auxiliary estimates}
	We first derive several lower bound estimates for the kernel $\mathcal{K}(t,x)$, introduced in Section \ref{sec2}. For convenience later, let us denote two domains with $t,r\ge 0$ as follows:
	$$\Gamma^1_{r,t}:= \{x\in\mathbb{R}^n:|x|\leq r\sqrt{t}\} \text{ and }\Gamma^2_{r,t}:= \{x\in\mathbb{R}^n:|x|\leq r\min\{\sqrt{t},t\}\}. $$
	\begin{lemma}\label{es_K}
		Let $1\leq n\leq3$ and fix $r_0\in (0,1)$. Then, $\mathcal{K}(t,x)$ is a nonnegative kernel, moreover, there exists a constant $C^0_1>0$ such that
		\begin{equation}\label{es1_K}
			\mathcal{K}(t,x)\geq C^0_1t^{-\frac{n}{2}} \quad \text{ for }t\geq 1, x\in \Gamma^1_{r_0,t}
		\end{equation}
		and
		\begin{equation}\label{es2_K}
			\mathcal{K}(t,x)\geq C^0_1(1+t)^{-\frac{n}{2}}\quad \text{ for }t>0, x\in \Gamma^2_{r_0,t}.
		\end{equation}
	\end{lemma}
	\begin{proof}
		Let us first prove \eqref{es1_K}. Set $r=|x|$, $\rho=\sqrt{t^2-r^2}$, and $z=\rho/2$. For $t\geq 1$, we have $r<t$ and $\rho\geq t/\sqrt{2}$. Then, it holds 
		\[
		\frac{t}{2}-z=\frac{t-\rho}{2}=\frac{t^2-\rho^2}{2(t+\rho)}=\frac{r^2}{2(t+\rho)}.
		\]
		Since $r^2\leq c^2t$ and $t+\rho\geq t$, we have $0\leq t/2-z\leq c^2/2$ which leads to
		\begin{equation}\label{es1_case_n1}
			e^{-t/2}e^z=\exp\tron{-\frac{r^2}{2(t+\rho)}}\geq e^{-c^2/2}.
		\end{equation}
		Using the representation formulas of the modified Bessel functions
		\[
		I_0(z)=\frac{1}{\pi}\int_0^{\pi}e^{z\cos\theta}\,d\theta\quad\text{and}\quad I_1(z)=\frac{1}{\pi}\int_0^{\pi}e^{z\cos\theta}\cos\theta\,d\theta
		\]
		we have the following asymptotic expansions:
		\begin{equation}\label{eq_approx_modified_Bessel}
			I_{\nu}(z)\sim \frac{e^z}{\sqrt{2\pi z}}\bigl(1+\mathcal{O}(z^{-1})\bigr)\qquad\text{as }z\to\infty,\quad \nu=0,1.
		\end{equation}
		\begin{itemize}[leftmargin=*]
			\item If $n=1$, then the kernel $\mathcal{K}$ has the form
			\[
			\mathcal{K}(t,x)=\frac{1}{2}e^{-t/2}
			I_0\left(\frac{\sqrt{t^2-|x|^2}}{2}\right)
			\mathbf{1}_{\{|x|<t\}}.
			\]
			Using \eqref{es1_case_n1}, \eqref{eq_approx_modified_Bessel} and $z\leq t/2$, we obtain
			\[
			\mathcal{K}(t,x)\gtrsim z^{-\frac{1}{2}}\gtrsim t^{-\frac{1}{2}}.
			\]
			\item If $n=2$, then the kernel $\mathcal{K}$ has the form
			\[
			\mathcal{K}(t,x)=\frac{e^{-t/2}}{2\pi}\frac{\cosh\left(\dfrac{\sqrt{t^2-|x|^2}}{2}\right)}{\sqrt{t^2-|x|^2}}\mathbf{1}_{\{|x|<t\}}.
			\]
			From \eqref{es1_case_n1} and the fact that $\cosh(z)\geq e^z/2$, we have
			\[
			\mathcal{K}(t,x)\gtrsim \rho^{-1}\gtrsim t^{-1}.
			\]
			\item If $n=3$, then the kernel $\mathcal{K}$ has the form
			\[
			\mathcal{K}(t,x)=e^{-t/2}\Bigg[\frac{1}{4\pi t}\delta_{\{|x|=t\}}+\frac{1}{8\pi}\frac{I_1\Big(\dfrac{\sqrt{t^2-|x|^2}}{2}\Big)}{\sqrt{t^2-|x|^2}}\mathbf{1}_{\{|x|<t\}}
			\Bigg].
			\]
			Here $\delta_{\{|x|=t\}}$ denotes the surface Dirac measure on the sphere
			$\{|x|=t\}$, which gives
			\[
			\left(\frac{1}{4\pi t}\delta_{\{|x|=t\}}\ast_x g\right)(x)=\int_{|y|=t}\frac{1}{4\pi t}g(x-y)\,dS_y.
			\]
			On the region $r\leq c\sqrt{t}<t$, it is clear to see that the term $(4\pi t)^{-1}\delta_{\{|x|=t\}}$ vanishes. Using \eqref{es1_case_n1} and \eqref{eq_approx_modified_Bessel} again, we obtain
			\[
			\mathcal{K}(t,x)\gtrsim e^{-t/2}e^z/(\rho \sqrt{z}),\text{ i.e. }\mathcal{K}(t,x)\gtrsim \rho^{-\frac{3}{2}}\gtrsim t^{-\frac{3}{2}}.
			\]
		\end{itemize}
		Therefore, the proof of \eqref{es1_K} is established. To prove \eqref{es2_K}, we note for $x\in \Gamma^2_{r_0,t}$ that $|x|^2\leq r^2_0 t$ if $t\geq 1$ and $|x|^2\leq r^2_0t^2$, $\rho\geq t\sqrt{1-r^2_0}$ if $0<t\leq 1$. This means that $t/2-z\leq r_0^2/2,$ and the estimate \eqref{es1_case_n1} still holds. Now for $z\geq 0$, we have
		\begin{align*}
			I_{0}(z)\geq Ce^z(1+z)^{-\frac{1}{2}};\;\quad \cosh(z)\geq Cze^z(1+z)^{-1};\;\quad I_{1}(z)\geq Cze^z(1+z)^{-\frac{3}{2}},
		\end{align*}
		which give
		\[
		\mathcal{K}(t,x)\gtrsim e^{-t/2+z}(1+z)^{-n/2} \text{ for }n=1,2,3.
		\]
		Hence, combining this with \eqref{es1_case_n1} we obtain \eqref{es2_K}.
	\end{proof}
	\noindent Next, we establish a lower-bound estimate for the solution to \eqref{Linear_damped_wave} in the following lemma.
	\begin{lemma}\label{lem_lowerbound_heatkernel}
		Let us assume $w_1\in L^1\cap H^1$ and denote the quantity 
		\[
		M_w:=\int_{\mathbb{R}^n}w_1(x)\,dx>0.
		\]
		Then, there exist $C^0_2>0$ and $T_0=T_0(w_1,M_w)>0$ such that the solution to \eqref{Linear_damped_wave} with $w_0=0$ satisfies
		\[
		w(t,x)=\mathcal{K}(t,x)\ast_x w_1(x)\geq C^0_2M_wt^{-\frac{n}{2}}
		\]
		for any $t\geq T_0$ and a.e $x\in\Gamma^1_{r_0,t}$. 
		
	\end{lemma}
	\begin{proof}
		Define $G(t,x):=(4\pi t)^{-n/2}e^{-|x|^2/(4t)}$. We have
		\[
		w(t,x)\geq G(t,\cdot)\ast_x w_1(x)-\|(w-G\ast_x w_1)(t)\|_{L^{\infty}}.
		\]
		Thanks to the diffusion phenomenon (see \cite{Marcati2003} for $n=1$ and \cite{Narazaki2004} for $n=2,3$), we obtain the estimate
		\begin{equation}\label{eq1_lem_lowerbound_heatkernel}
			\|(w-G\ast_x w_1)(t)\|_{L^{\infty}}\leq C t^{-\frac{n}{2}-1}\|w_1\|_{L^1\cap H^1}
		\end{equation}
		for $t\geq T^0_1$, where $T^0_1$ is a sufficiently large positive constant. On the other hand, we write
		\[
		G(t,\cdot)\ast_x w_1(x)=(4\pi t)^{-\frac{n}{2}}e^{-\frac{|x|^2}{4t}}I_t(x),
		\]
		where
		\[
		I_t(x):=\int_{\mathbb R^n}\exp\left(\frac{x\cdot y}{2t}-\frac{|y|^2}{4t}\right)w_1(y)\,dy.
		\]
		We claim that
		\begin{equation}
			\label{uniform_mass_convergence}
			\sup_{x\;\in \;\Gamma^1_{r_0,t}}|I_t(x)-M_w|\to 0\qquad\text{as }t\to\infty,
		\end{equation}
		where
		\[
		M_w:=\int_{\mathbb R^n}w_1(y)\,dy>0.
		\]
		Indeed, one has
		\[
		\sup_{x\in \Gamma^1_{r_0,t}}|I_t(x)-M_w|\leq\int_{\mathbb R^n}\sup_{x\in \Gamma^1_{r_0,t}}\left|\exp\left(\frac{x\cdot y}{2t}-\frac{|y|^2}{4t}\right)-1\right||w_1(y)|\,dy.
		\]
		For any fixed $y\in\mathbb R^n$, it follows that
		\[
		\sup_{x\;\in\;\Gamma^1_{r_0,t}}\left|\frac{x\cdot y}{2t}-\frac{|y|^2}{4t}\right|\leq\frac{r_0|y|}{2\sqrt{t}}+\frac{|y|^2}{4t}\to 0
		\]
		as $t\to\infty$. As a consequence, we arrive at
		\[
		\sup_{x\;\in\;\Gamma^1_{r_0,t} }\left|\exp\left(\frac{x\cdot y}{2t}-\frac{|y|^2}{4t}\right)-1\right|\to 0.
		\]
		On the other hand, for $x\in \Gamma^1_{r_0,t}$ one obtains
		\[
		\frac{x\cdot y}{2t}-\frac{|y|^2}{4t}=\frac{|x|^2}{4t}-\frac{|x-y|^2}{4t}\leq \frac{r_0^2}{4},
		\]
		which implies
		\[
		\sup_{x\;\in\;\Gamma^1_{r_0,t}}\left|\exp\left(\frac{x\cdot y}{2t}-\frac{|y|^2}{4t}\right)-1\right|\leq 1+e^{r_0^2/4}.
		\]
		Since $w_1\in L^1$, the dominated convergence theorem yields \eqref{uniform_mass_convergence}. For $M_w>0$, there exists $T^0_2>0$ such that $I_t(x)\geq M_w/2$ for any $t\geq T^0_2$ and $x\in \Gamma^1_{r_0,t}$. Moreover, it is obvious to see that
		\[
		e^{-\frac{|x|^2}{4t}}\geq e^{-r_0^2/4}
		\]
		on the same region. Hence, we may conclude
		\begin{equation}\label{lowerbound_G_w1}
			G(t,\cdot)\ast_x w_1(x)\geq\frac{1}{2}(4\pi)^{-\frac{n}{2}}e^{-r_0^2/4} M_wt^{-\frac{n}{2}},
		\end{equation}
		for $t\geq T^0_2$ and $x\in \Gamma^1_{r_0,t}$. Set
		\[
		C^0_2:=\frac{1}{4}(4\pi)^{-\frac n2}e^{-r_0^2/4}.
		\]
		Next, for $T^0_1$ is large enough, we deduce from \eqref{eq1_lem_lowerbound_heatkernel} the following estimate:
		$$ \|(w-G\ast_x w_1)(t)\|_{L^{\infty}} \le C^0_2M_wt^{-n/2}. $$
		for $t\geq T^0_1$. In this way, from \eqref{lowerbound_G_w1} we arrive at
		\[
		w(t,x)\geq C^0_2M_wt^{-\frac n2}
		\]
		for $t\geq T_0:=\max\{T^0_1,T^0_2\}$ and a.e $x\in\Gamma^1_{r_0,t}$. 
	\end{proof}
	\subsection{Proof of Theorem \ref{Thr1_Blowup}}
	Assume that $U=(u,v)$ is a global (in time) mild solution to \eqref{Main_system} with initial data satisfying $u_0=v_0=0$. Then, for $w\in\{u,v\}$, the identity
	\begin{align*}
		w(t,x)=\mathcal{K}(t,x)\ast_xw_1(x)+\int_{0}^{t}\mathcal{K}(t-s,x)*\bigl(|u(s,x)|^{\alpha_w}|v(s,x)|^{\beta_w}\bigr)\,ds 
	\end{align*}
	holds in $\mathcal{C}([0,\infty);L^2)$. We want to stress out that each mild solution with a measurable representative is defined a.e by the right-hand side of the Duhamel's principle. All the pointwise inequalities below are based on iteration argument. To get started, let us fix the constant $C^0:=\min\{C^0_1,C^0_2\}$. Since $\mathcal{K}$ is nonnegative, the nonlinear terms satisfy
	\begin{equation}\label{eq_nonlinearterm_non-negative}
		w^{\rm nl}(t,x)=\int_{0}^t\mathcal{K}(t-s,x)\ast_x(|u(s,\cdot)|^{\alpha_w}|v(s,\cdot)|^{\beta_w})ds\geq 0
	\end{equation}
	for $w\in\{u,v\}$. By Lemma \ref{lem_lowerbound_heatkernel} together with \eqref{eq_nonlinearterm_non-negative}, there exists a constant $T_0$ such that 
	\[
	w(t,x)\geq C^0M_{w}t^{-\frac{n}{2}}
	\]
	for $t\geq T_0$, a.e $x\in\Gamma^1_{r_0,t}$ and $w\in\{u,v\}$.
	
	\medskip
	\noindent{\bf $\bullet$ Step 1:} It follows from Duhamel's formula that 
	\begin{equation}\label{es_Duhamel_component}
		w(t,x)\geq C^0M_{w}t^{-\frac{n}{2}}+\int_{(1-\delta_0)t}^{t}\int_{\{y\in \mathbb R^n:\;x-y\;\in\;\Gamma^2_{r_0,t-s}\}}\mathcal{K}(t-s,x-y)|u(s,y)|^{\alpha_w}|v(s,y)|^{\beta_w}\,dy\,ds
	\end{equation}
	for $t\geq T_0$ and a.e $x\in\Gamma^1_{r_0,t}$, where $\delta_0\in(0,1)$ will be fixed later. The application of Lemma \ref{es_K} gives
	\begin{equation}\label{es_K(t-s,x-y)}
		\mathcal{K}(t-s,x-y)\geq C^0(1+t-s)^{-\frac{n}{2}},  
	\end{equation}
	provided that $x-y\in \Gamma^2_{r_0,t-s}$. Now, for $x\in \Gamma^1_{r_1,t}$ with $r_1\leq r_0$, $x-y\in\Gamma^2_{r_0,t-s}$ and $s\geq (1-\delta_0)t$, we have
	\begin{equation}\label{es_domain_y_Thrblowup1}
		|y|\leq |x|+|x-y|\leq r_1\sqrt{t}+r_0\sqrt{t-s}\leq (r_1+r_0\sqrt{\delta_0})\sqrt{t}. 
	\end{equation}
	Choosing $r_1=r_0\sqrt{1-\delta_0}-r_0\sqrt{\delta_0}\in(0,1)$, we obtain from \eqref{es_domain_y_Thrblowup1} that 
	\begin{equation}\label{es_domain_y}
		|y|\leq r_0\sqrt{1-\delta_0}\sqrt{t}\leq r_0\sqrt{s}.
	\end{equation}
	Moreover, by choosing $T_1=T_0/(1-\delta_0)$ it holds
	\begin{equation}\label{es_domain_s}
		s\geq (1-\delta_0)t\geq (1-\delta_0)T_1=T_0
	\end{equation}
	for $t\geq T_1$. From \eqref{es_domain_y}, \eqref{es_domain_s} and Lemma \ref{lem_lowerbound_heatkernel}, for any $w\in\{u,v\}$ we have
	\begin{equation}\label{es_lowerbounded_nonlinearity}
		|u(s,y)|^{\alpha_w}|v(s,y)|^{\beta_w}
		\geq (C^0M_u)^{\alpha_w}(C^0M_v)^{\beta_w}
		s^{-\frac n2(\alpha_w+\beta_w)}.
	\end{equation}
	Substituting \eqref{es_K(t-s,x-y)} and \eqref{es_lowerbounded_nonlinearity} into \eqref{es_Duhamel_component}, we obtain
	\begin{equation}\label{es_component_Step0}
		w(t,x)\geq (C^0)^{1+\alpha_w+\beta_w}M^{\alpha_w}_{u} M^{\beta_w}_{v}\int_{(1-\delta_0)t}^{t}\int_{\{y\in \mathbb R^n:\;x-y\;\in\;\Gamma^2_{r_0,t-s}\}}(1+t-s)^{-\frac{n}{2}}s^{-\frac{n}{2}(\alpha_w+\beta_w)}\,dy\,ds.
	\end{equation}
	Moreover, the volume of $\Gamma^2_{r_0,t-s}$ satisfies
	\[
	\int_{\{y\in \mathbb R^n:\;x-y\;\in\;\Gamma^2_{r_0,t-s}\}}\,dy\geq V\min\{t-s,\sqrt{t-s}\}^{n},
	\]
	where $V=V(n,r_0)>0$ is a suitable constant. By the change of variables $\tau=t-s$, we set
	\[
	W(\tau):=(1+\tau)^{-\frac{n}{2}}\min\{\tau,\sqrt{\tau}\}^{n}
	\]
	to gain
	\[
	\begin{aligned}
		w(t,x)&\geq (C^0)^{1+\alpha_w+\beta_w} M^{\alpha_w}_{u} M^{\beta_w}_{v} V\int_{0}^{\delta_0t}W(\tau)(t-\tau)^{-\frac{n}{2}\alpha_w-\frac{n}{2}\beta_w}\,d\tau
	\end{aligned}
	\]
	from \eqref{es_component_Step0} for $w\in\{u,v\}$. The integral on the right-hand side satisfies
	\[
	\int_{0}^{\delta_0t}W(\tau)(t-\tau)^{-\frac{n}{2}\alpha_w-\frac{n}{2}\beta_w}d\tau\geq C_{W}\delta_0^{n+1}t^{1-\frac{n}{2}\alpha_w-\frac{n}{2}\beta_w}
	\]
	for a constant $C_W= C_W(t,\delta_0)$. Therefore, we arrive at
	\[
	w(t,x)\geq C(C^0)^{\alpha_w+\beta_w}M^{\alpha_w}_{u}M^{\beta_w}_{v}\delta_0^{n+1}t^{1-\frac{n}{2}\alpha_w-\frac{n}{2}\beta_w}
	\]
	for $t\geq T_1$ and a.e $x\in\Gamma^1_{r_1,t}$, where $C=C^0VC_{W}$. 
	
	\medskip
	\noindent{\bf $\bullet$ Step $k$:} Let us define the following sequences with $k\ge 0$: 
	\begin{align*}
		&a_0=b_0= n/2, \quad c_0=C^0M_u,\quad d_0=C^0M_v,\quad \delta_0=\delta^*/(1+n),\\
		&a_{k+1}=\min\{a_k,a_{k}\alpha+b_{k}p-1\},\\
		&b_{k+1}=\min\{b_k,a_{k}q+b_{k}\beta-1\}, \\
		&c_{k+1}=\left\{
		\begin{aligned}
			&C(c_{k})^{\alpha}(d_{k})^{p}\delta^{n+1}_k&&\text{if }a_{k+1}<a_k,\\
			&c_{k}&&\text{if }a_{k+1}=a_{k},
		\end{aligned}
		\right.\\
		&d_{k+1}=\left\{
		\begin{aligned}
			&C(c_{k})^{q}(d_{k})^{\beta}\delta^{n+1}_{k}&&\text{if }b_{k+1}<b_{k},\\
			&d_k &&\text{if }b_{k+1}=b_k,
		\end{aligned}
		\right. \\
		&r_{k+1}=r_k\sqrt{1-\delta_k}-r_0\sqrt{\delta_k}\quad\text{and}\quad T_{k+1}=\frac{T_k}{1-\delta_k}
	\end{align*}
	with
	$$ \delta_k=\frac{\delta^*}{1+|a_k|+|b_k|}, $$ where $\delta^*$ will be chosen later. Now we assume that $u(t,x)\geq c_kt^{-a_k}$ and $v(t,x)\geq d_kt^{-b_k}$ for $t\geq T_k$ and a.e $x\in\Gamma^1_{r_k,t}$. For $w\in\{u,v\}$, set $\lambda_{w,k}:=\alpha_w a_k+\beta_w b_k$. For $t\geq T_{k+1}$, $s\in[(1-\delta_k)t,t]$, and $x\in\Gamma^1_{r_{k+1},t}$, the choice of $r_{k+1}$ and $T_{k+1}$ implies that $y\in \Gamma^1_{r_k,s}$ and $s\geq T_k$. 
	Repeating an argument as we did in Step 1, we obtain
	\[
	w(t,x)
	\geq Cc_k^{\alpha_w}d_k^{\beta_w}
	\int_0^{\delta_kt}W(\tau)(t-\tau)^{-\lambda_{w,k}}\,d\tau.
	\]
	Next, if we choose $\delta^*\leq 1/4$, since $\delta_k\leq \delta^*\leq 1/4$, it leads to $(t-\tau)^{-\lambda_{w,k}}
	\geq C^0_1t^{-\lambda_{w,k}}$ for $0\leq\tau\leq\delta_kt$, where $C^0_1>0$ is independent of $w$ and $k$. Finally, for any $t\geq1$, we get
	\[
	\int_0^{\delta_k t}W(\tau)\,d\tau
	\geq C_W\delta_k^{n+1}t,
	\]
	where $C_W>0$ is independent of $t$ and $\delta_k$, by considering separately the cases $\delta_k t\leq1$ and $\delta_k t>1$. Consequently, we obtain
	\[
	w(t,x) 
	\geq Cc_k^{\alpha_w}d_k^{\beta_w}
	\delta_k^{n+1}t^{1-\lambda_{w,k}},
	\]
	for $w\in \{u,v\}$, with $C>0$ independent of $k$. Recalling the recursive definitions of $a_{k+1},b_{k+1},c_{k+1}$, and $d_{k+1}$ we claim the following estimates for any $k\in\mathbb N$:
	\begin{equation} \label{Relation_a_k-b_k}
		u(t,x)\geq c_kt^{-a_k},
		\qquad
		v(t,x)\geq d_kt^{-b_k},
	\end{equation}
	with $t\geq T_k$ and a.e $x\in\Gamma^1_{r_k,t}$, provided that $r_k>0$. \medskip
	
	Next, let us show below that
	$\delta^*$ can be chosen sufficently small so that $r_k\geq r_0/2>0$ for all $k\in \mathbb N$. Let $A_k=(a_k,b_k)^{\top}$ and define
	\[
	P=\begin{pmatrix}\alpha&p\\ q&\beta\end{pmatrix}.
	\]
	From the definitions of $a_{k+1}$ and $b_{k+1}$, it holds $A_{k+1}=\min\{A_k,PA_k-\boldsymbol{1}\}$. If $\{A_k\}_{k\in \mathbb N}$ is bounded below, then $A_k\to A_{\infty}\in\mathbb{R}^2$ as $k\to\infty$. Here, $A_{\infty}\leq A_{0}$ and  
	\[
	A_{\infty}=\min\{A_{\infty},PA_{\infty}-\boldsymbol{1}\}.
	\]
	Hence, $A_{\infty}\leq PA_{\infty}-\boldsymbol{1}$, which leads to $(P-I)A_{\infty}\geq \boldsymbol{1}$. The following lemma shows that such a vector $A_{\infty}$ cannot exist.
	\begin{lemma}\label{lem_matrix}
		Let
		\[
		L=P-I=\begin{pmatrix}\alpha-1&p\\ q&\beta-1\end{pmatrix}
		\]
		and assume that \eqref{Thr1blowup_Criticalcurve} holds. Then there is no $\Theta=(\theta_1,\theta_2)^{\top}$ with $\theta_i\leq n/2$ for $i=1,2$ such that
		\[
		L\Theta\geq \boldsymbol{1}.
		\]
	\end{lemma}
	
	\begin{proof}
		Suppose that there exists $\Theta=(\theta_1,\theta_2)^{\top}$ with $\theta_i\leq n/2$ for $i=1,2$ such that $L\Theta\geq \boldsymbol{1}$, i.e.
		\begin{align}
			-(1-\alpha)\theta_1+p\theta_2&\geq 1,\label{eq1_lem_matrix}\\
			q\theta_1-(1-\beta)\theta_2&\geq 1.\label{eq2_lem_matrix}
		\end{align}
		Multiplying both sides of \eqref{eq1_lem_matrix} by $1-\beta$ and both sides of \eqref{eq2_lem_matrix} by $p$ and then adding the resulting inequalities, we obtain
		\[
		\theta_1\geq \frac{p+1-\beta}{pq-(1-\alpha)(1-\beta)}.
		\]
		Similarly, we also derive
		\[
		\theta_2\geq \frac{q+1-\alpha}{pq-(1-\alpha)(1-\beta)}.
		\]
		Combining the above estimates one finds
		\[
		\max\{\theta_1,\theta_2\}\geq \frac{1+\max\{q-\alpha,p-\beta\}}{pq-(1-\alpha)(1-\beta)}>\frac{n}{2},
		\]
		thanks to \eqref{Thr1blowup_Criticalcurve}. It is a contradiction to the assumption of $\theta_i\leq n/2$ for $i=1,2$. Hence, this completes the proof of Lemma \ref{lem_matrix}.
	\end{proof}
	By Lemma \ref{lem_matrix}, at least one of the sequences $\{a_k\}_{k\in\mathbb N}$ and $\{b_k\}_{k\in \mathbb N}$ are not bounded below. The recursive relations then imply that both $a_k\to-\infty$ and $b_k\to-\infty$ as $k\to\infty$. We define
	\begin{equation}\label{def_X_k_Y_k}
		X_k(t)=\log(c_k)-a_k\log(t),\qquad Y_k(t)=\log(d_k)-b_k\log(t).
	\end{equation}
	By \eqref{Relation_a_k-b_k}, we have
	\begin{equation}\label{es_u_v}
		u(t,x)\geq e^{X_k(t)} \text{ and }v(t,x)\geq e^{Y_k(t)} 
	\end{equation}
	for $t\geq T_k$ and a.e $x\in \Gamma^1_{r_k,t}$. Let $R_k=(-a_k,-b_k)^{\top}$ and $H_k=(-\log(c_k),-\log(d_k))^{\top}$. We show that there exist a $2\times 2$ matrix $M_k$ and $2\times 1$ vectors $E_k,\eta_k$ such that
	\begin{align*}
		R_{k+1}=M_kR_k+\eta_k\quad\text{and}\quad H_{k+1}\leq M_kH_k+E_k.
	\end{align*}
	Indeed, we set
	\begin{align*}
		&a_{k+1}=(1-\varepsilon_{u,k})a_k+\varepsilon_{u,k}\left(\alpha a_k+pb_k-1\right),\\
		&b_{k+1}=(1-\varepsilon_{v,k})b_k+\varepsilon_{v,k}\left(qa_k+\beta b_k-1\right),
	\end{align*}
	where $\varepsilon_{u,k},\,\varepsilon_{v,k}\in\{0,1\}$, and define the following matrices:
	\[
	M_k:=\begin{pmatrix}
		1-\varepsilon_{u,k}+\varepsilon_{u,k}\alpha
		&
		\varepsilon_{u,k}p
		\\[2mm]
		\varepsilon_{v,k}q
		&
		1-\varepsilon_{v,k}+\varepsilon_{v,k}\beta
	\end{pmatrix}
	\quad\text{ and }\quad
	\eta_k:=
	\begin{pmatrix}
		\varepsilon_{u,k}\\
		\varepsilon_{v,k}
	\end{pmatrix}.
	\]
	Then, $R_{k+1}=M_kR_k+\eta_k$. Next we consider the sequences $\{c_k\}_{k\in \mathbb N}$ and
	$\{d_k\}_{k\in \mathbb N}$ to express
	\begin{align*}
		c_{k+1}=c_k^{1-\varepsilon_{u,k}}
		\left(C c_k^{\alpha}d_k^{p}
		\delta_k^{n+1}
		\right)^{\varepsilon_{u,k}},\qquad d_{k+1}=d_k^{1-\varepsilon_{v,k}}
		\left(C c_k^{q}d_k^{\beta}
		\delta_k^{n+1}
		\right)^{\varepsilon_{v,k}},
	\end{align*}
	which gives the expressions as follows:
	\[
	\begin{aligned}
		-\log(c_{k+1}) &=\bigl(1-\varepsilon_{u,k}+\varepsilon_{u,k}\alpha\bigr)
		\bigl(-\log(c_k)\bigr)+\varepsilon_{u,k}p
		\bigl(-\log(d_k)\bigr)+\varepsilon_{u,k}\Lambda_k, \\
		-\log(d_{k+1}) &=\varepsilon_{v,k}q
		\bigl(-\log(c_k)\bigr)+\bigl(1-\varepsilon_{v,k}
		+\varepsilon_{v,k}\beta\bigr)
		\bigl(-\log(d_k)\bigr)+\varepsilon_{v,k}\Lambda_k,
	\end{aligned}
	\]
	where $\Lambda_k:=-\log C-(n+1)\log\delta_k$. By the definition of $H_k$, we have the identity $H_{k+1}=M_kH_k+\Lambda_k\eta_k$. We define $E_k:=[\Lambda_k]^{+}\eta_k$. Since $\Lambda_k\eta_k\leq[\Lambda_k]^{+}\eta_k$, we conclude that $H_{k+1}\leq M_kH_k+E_k$. Next, we derive an upper bound for $E_k$. Let $S_k=1+|a_k|+|b_k|$. The recursive relations for $a_k$ and $b_k$ give
	\[
	|a_{k+1}|\leq (1+p+\alpha)S_k\; \text{  and  }\;|b_{k+1}|\leq (1+q+\beta)S_k.
	\]
	These estimates imply that
	\[
	S_{k+1}=1+|a_{k+1}|+|b_{k+1}|\leq (3+p+\alpha+q+\beta)S_k.
	\]
	Therefore, we obtain the estimate $S_{k}\leq \rho_1^kS_0=\rho_1^k(1+n)$ with $\rho_1:=3+p+\alpha+q+\beta>1$. It follows that
	\[
	-\log C- (n+1)\log(\delta_k)\leq C\bigl(1+\log(1+|a_k|+|b_k|)\bigr)\leq C(k+1).
	\]
	Therefore, we arrive at
	\begin{equation}\label{es_for_E}
		E_k\leq C(k+1)\boldsymbol{1}.
	\end{equation}
	On the other hand, we have $R_{k+1}\geq PR_k+\boldsymbol{1}$. Since both $a_k\to-\infty$ and $b_k\to-\infty$ as $k\to\infty$, there exists $k_0\geq1$ such that
	\[
	R_{k_0}=(-a_{k_0},-b_{k_0})^{\top}\geq C\boldsymbol{1}.
	\]
	Set $\rho_2:= \min\{p+\alpha,q+\beta\}>1$. Then $P\boldsymbol{1}\geq \rho_2\boldsymbol{1}$ and $R_{k_0+1}\geq PR_{k_0}\geq CP\boldsymbol{1}\geq C\rho_2\boldsymbol{1}$. By an induction argument, it holds $R_{k_0+l}\geq C\rho^l_2\boldsymbol{1}$ for $l\in\mathbb{N}$. For $k\geq k_0$, choosing $l=k+1-k_0$ gives
	\begin{equation}\label{es_for_R}
		R_{k+1}\geq (C\rho^{1-k_0}_2)\rho^k_2\boldsymbol{1}.
	\end{equation}
	It follows from \eqref{es_for_E}-\eqref{es_for_R} that there exists a nonegative sequence $\{\varepsilon_k\}_{k\geq1}$ satisfying $E_k\leq\varepsilon_kR_{k+1}$ and
	\begin{equation}\label{sum_varepsilon}
		\sum_{k=0}^{\infty}\varepsilon_{k}<\infty.
	\end{equation}
	Choose $Q_{k_0}\geq0$ such that $H_{k_0}\leq Q_{k_0}R_{k_0}$ and let $Q_{k+1}=Q_{k}+\varepsilon_k$ with $k\geq k_0$. If $H_{k}\leq Q_{k} R_{k}$, then the relation
	\[
	H_{k+1}\leq M_{k}H_{k}+E_{k}\leq Q_kM_kR_k+\varepsilon_kR_{k+1}\leq (Q_{k}+\varepsilon_k)R_{k+1}
	\]
	is valid, where the last inequality follows from the fact that $M_{k}R_{k}=R_{k+1}-\eta_{k}\leq R_{k+1}$. By an induction argument, it holds $H_k\leq Q_kR_k$ for all $k\geq k_0$. Moreover, \eqref{sum_varepsilon} implies that $Q_k$ is bounded above. Hence, there exists $Q\in \mathbb{R}$ such that $H_k\leq QR_k$ for all $k\geq k_0$. This leads to
	\begin{equation}\label{final_es_blowup}
		-\log(c_k)\leq Q(-a_k)\quad\text{and}\quad-\log(d_k)\leq Q(-b_k).
	\end{equation}
	Substituting \eqref{final_es_blowup} into \eqref{def_X_k_Y_k}, we obtain $X_k(t)\geq-a_k(\log t-Q)$ and $Y_k(t)\geq-b_k(\log t-Q)$ for $t\geq T_k$ and a.e $x\in\Gamma^1_{r_k,t}$. Again, using an induction argument we obtain $r_k\leq r_0$ for all $k\in\mathbb{N}$. Thus, it entails
	\[
	r_k-r_{k+1}=r_k(1-\sqrt{1-\delta_k})+r_0\sqrt{\delta_k}\leq r_0\delta_k+r_0\sqrt{\delta_k},
	\]
	which gives
	\[
	r_k\geq r_0-r_0\sum_{j=0}^{k-1}(\delta_j+\sqrt{\delta_j}).
	\]
	Since $\rho_2>1$, we have $|a_j|+|b_j|\gtrsim \rho_2^j$ for all sufficiently large $j$. Then, one has
	\[
	\sum_{j=0}^{\infty}\bigg(\frac{1}{1+|a_j|+|b_j|}+\frac{1}{\sqrt{1+|a_j|+|b_j|}}\bigg)<\infty,
	\]
	which guarantees that we can choose $\delta^*<1/8$ fulfilling
	\[
	\sum_{j=0}^{\infty}(\delta_j+\sqrt{\delta_j})<\frac{1}{2}
	\]
	and then $r_k\geq r_0/2$ for all $k\in\mathbb{N}$. Hence, for $t>T_k>T_0>0$, the regions $\Gamma^1_{r_k,t}$ remain nonempty for all $k\in \mathbb{N}$.
	
	Finally, since $0\leq \delta_k\leq\delta^*\leq 1/8$, we have $-\log(1-\delta_k)\leq \delta_k/(1-\delta_k)\leq 8/7\delta_k$. Recalling
	\[
	T_k=T_0\prod_{j=0}^{k-1}\frac{1}{1-\delta_j}
	\]
	we get
	\[
	\log\tron{\frac{T_k}{T_0}}= -\sum_{j=0}^{k-1}\log(1-\delta_j)\leq\frac{8}{7}\sum_{j=0}^{\infty}\delta_j<\infty.
	\]
	Consequently, there exists $T_{\infty}>0$ such that $T_k\leq T_\infty$ for all $k\in\mathbb{N}$. Choosing $t^*>T_{\infty}$ such that $\log(t^*)-Q>0$, we arrive at 
	\begin{equation}\label{X,Y_to_infty}
		X_k(t^*),Y_k(t^*)\to\infty\qquad\text{as }k\to\infty.
	\end{equation}
	By \eqref{es_u_v}, we may estimate the following $L^2$-lower bounds:
	\begin{equation}\label{lowerbound_L2_norm}
		\|u(t^*,\cdot)\|_{L^2}\gtrsim \big|\Gamma^1_{r_0/2,t^*}\big|^{1/2}e^{X_k(t^*)},\qquad \|v(t^*,\cdot)\|_{L^2}\gtrsim \big|\Gamma^1_{r_0/2,t^*}\big|^{1/2}e^{Y_k(t^*)}.
	\end{equation}
	Summarizing, the combination of \eqref{X,Y_to_infty} and \eqref{lowerbound_L2_norm} yields a contradiction, hence, our proof is complete.
	
	\section{Further discussions and final remarks}\label{sec4}
	
	In this subsection, we briefly discuss the case in which the exponents of the pertubation terms satisfy $\alpha,\beta\geq 1$. In this case, the thresholds are governed by the total degrees $p+\alpha$ and $q+\beta$ of the nonlinearities in the two equations, rather than by the coupled critical curve. Moreover, if $p,q\geq 1$, the nonlinear mappings $(u,v)\to |u|^{\alpha}|v|^{p}$ and $(u,v)\to |u|^{q}|v|^{\beta}$ for any $\alpha,\beta\geq 1$ are locally Lipschitz in the solution space (see Lemma \ref{Lem_Lipschitz}). Therefore, the compactness argument based on Schauder's fixed point theorem is no longer needed, and the global solution can be constructed directly by Banach's contraction principle which yields the uniqueness of the resulting mild solution. More precisely, we have the following result.
	\begin{proposition}\label{Thr_2_Globalexistence}
		Let $p,q>0$ and $\alpha,\beta\in[1,\infty)$. Set $r=\min\{2,p+\alpha,q+\beta\}$ and $\kappa_{r}=(n-1)|1/2-1/r|$. Assume that $1\leq n\leq3$ together with \eqref{con_G-N_Thr1} fulfilling the following condition:
		\begin{equation}\label{Thr2globalexistence_Critical_exponent}
			\min\{p+\alpha,q+\beta\}>1+2/n.
		\end{equation}
		Then there exists a constant $\varepsilon_0>0$ such that for any small initial data $(u_0,u_1,v_0,v_1)\in \mathcal{D}_r:=\bigl(H^1\cap H^{\kappa_r}_r\cap L^1\bigr)^2\times\bigl(L^2\cap L^1\bigr)^2$ satisfying $\|(u_0,u_1,v_0,v_1)\|_{\mathcal{D}_r}\leq\varepsilon_0$, the problem \eqref{Main_system} admits a global (in time) mild solution in the sense of Definition \ref{MildSol.Def} belonging to
		\begin{align*}
			(u,v)\in (\mathcal{C}([0,\infty); H^1\cap L^r)\cap \mathcal{C}^1([0,\infty);L^2))^2.
		\end{align*}
		Moreover, the following estimates hold:
		\begin{align*}
			\|u(t,\cdot)\|_{\dot{H}^1}+\|v(t,\cdot)\|_{\dot{H}^1}&\lesssim (1+t)^{-\frac{n}{4}-\frac{1}{2}} \|(u_0,u_1,v_0,v_1)\|_{\mathcal{D}_r},\\
			\|u(t,\cdot)\|_{L^{2}}+\|v(t,\cdot)\|_{L^2}&\lesssim (1+t)^{-\frac{n}{4}}\|(u_0,u_1,v_0,v_1)\|_{\mathcal{D}_r},\\
			\|u(t,\cdot)\|_{L^{r}}+\|v(t,\cdot)\|_{L^r}&\lesssim (1+t)^{-\frac{n}{2}\tron{1-\frac{1}{r}}}\|(u_0,u_1,v_0,v_1)\|_{\mathcal{D}_r},\\
			\|\partial_tu(t,\cdot)\|_{L^2}+\|\partial_tv(t,\cdot)\|_{L^2}&\lesssim (1+t)^{-\frac{n}{4}-1}\|(u_0,u_1,v_0,v_1)\|_{\mathcal{D}_r}.
		\end{align*}
	\end{proposition}
	
	\begin{proof}
		The case of $p\in (0,1)$ or $q\in (0,1)$ follows by the same compactness argument as in the proof of Theorem \ref{Thr_1_Globalexistence}, since the nonlinearities are not locally Lipschitz at the origin. Therefore, we omit the proof in detail to avoid the repetition. Instead, we focus on the case $p,q\geq 1$. Here the nonlinearities are locally Lipschitz, so that the Banach fixed point theorem can be employed, yielding a unique global (in time) mild solution. For each $T\geq0$, we introduce the solution space
		\[
		X(T):=\bigl(\mathcal{C}([0,T];H^1\cap L^r)\cap \mathcal{C}^1([0,T];L^2)\bigr)^2
		\]
		with the weighted norm defined as follows:
		\[
		\begin{aligned}
			\|U\|_{X(T)}:=\sup_{s\in[0,T]}\sum_{w\in\{u,v\}}\Big(f_1(s)\|w(s,\cdot)\|_{L^r}+
			f_2(s)\|w(s,\cdot)\|_{L^2}
			+f_3(s)\|w(s,\cdot)\|_{\dot H^1}
			+f_4(s)\|\partial_tw(s,\cdot)\|_{L^2}\Big).
		\end{aligned}
		\]
		where $f_1(s)=(1+s)^{\frac{n}{2}\tron{1-\frac{1}{r}}}$,\;$f_2(s)=(1+s)^{\frac{n}{4}}$,\;$f_3(s)=(1+s)^{\frac{n}{4}+\frac{1}{2}}$,\;$f_4(s)=(1+s)^{\frac{n}{4}+1}$. For $U=(u,v)$ and $\overline U=(\overline u,\overline v)$, our aim is to prove the following pair of inequalities:
		\begin{align}
			\|\mathcal{N}[U]\|_{X(T)} &\leq C_0\big(\|(U_0,U_1)\|_{\mathcal{D}}+\|U\|^{p+\alpha}_{X(T)}+\|U\|^{q+\beta}_{X(T)}\big), \label{eqNu_Thr2} \\
			\|\mathcal{N}[U]-N[\overline{U}]\|_{X(T)} &\leq C_1\|U-\overline{U}\|_{X(T)}\tron{\|U\|^{p+\alpha-1}_{X(T)}+\|\overline{U}\|^{p+\alpha-1}_{X(T)}+\|U\|^{q+\beta-1}_{X(T)}+\|\overline{U}\|^{q+\beta-1}_{X(T)}}, \label{eqNu-Nv_Thr2}
		\end{align}
		where $\mathcal{N}$ is defined as in \eqref{mapN}, two constants $C_0$ and $C_1$ are independent of $T$. Choose $\varepsilon$ and $\varepsilon_0$ sufficiently small such that $\max\{2C_1,C_0\}\varepsilon^{p+\alpha-1}\leq 1/3$, $\max\{2C_1,C_0\}\varepsilon^{q+\beta-1}\leq 1/3$ and $\varepsilon_0\leq\varepsilon/(3C_0)$. It follows from \eqref{eqNu_Thr2} and \eqref{eqNu-Nv_Thr2} that
		\begin{align*}
			\|\mathcal{N}[U]\|_{X(T)}\leq \varepsilon \quad\text{ and }\quad
			\|\mathcal{N}[U]-N[\overline{U}]\|_{X(T)}\leq \frac{2}{3}\|U-\overline{U}\|_{X(T)}.
		\end{align*}
		Hence, $\mathcal{N}:X(T,\varepsilon)\to X(T,\varepsilon)$ is a contraction mapping, where
		\begin{equation*}\label{def_XT_varepsilon_Thr2}
			X(T,\varepsilon):=\{U\in X(T):\|U\|_{X(T)}\leq \varepsilon\},
		\end{equation*}
		provided that the initial data satisfy $\|(U_0,U_1)\|_{\mathcal{D}}\leq \varepsilon_0$. Therefore, for any $T>0$ there exists a unique fixed point $U^{(T)}=\mathcal{N}(U^{(T)}) \text{ in } X(T,\varepsilon)$. If $0<T_1<T_2$, the restriction of $U^{(T_2)}$ on $[0,T_1]$ is also a fixed point in $X(T_1,\varepsilon)$, hence, the uniqueness gives
		\[
		{\bf 1}_{[0,T_1]}(t)U^{(T_2)}=U^{(T_1)}.
		\]
		This means that the family $\{U^{(T)}\}_{T>0}$ defines a unique global (in time) solution $U\in X(\infty,\varepsilon)$. Moreover, from \eqref{eqNu_Thr2} and $U=\mathcal{N}[U]$, we arrive at
		\[
		\|U\|_{X(T)}\leq C_0\|(U_0,U_1)\|_{\mathcal D}+C_0\Bigl(\|U\|_{X(T)}^{p+\alpha-1}+\|U\|_{X(T)}^{q+\beta-1}\Bigr)\|U\|_{X(T)}.
		\]
		From the choice of $\varepsilon$, we get $\|U\|_{X(T)}\leq 3C_0\|(U_0,U_1)\|_{\mathcal D}$ uniformly in $T$. This yields the stated decay estimates. It remains to prove the following lemma.
		\begin{lemma}\label{lemma_G-N_Thr2globalexistence}
			Under the assumptions of Proposition \ref{Thr_2_Globalexistence}, the following estimates hold for $w\in\{u,v\}$ and $\gamma\in\{1,2,r\}$:
			\begin{align*}
				\big\||u(s,\cdot)|^{\alpha_w}|v(s,\cdot)|^{\beta_w}\big\|_{L^{\gamma}} &\lesssim  (1+s)^{-\frac{n}{2}\tron{\alpha_w+\beta_w-\frac{1}{\gamma}}} \|U\|_{X(T)}^{\alpha_w+\beta_w},\\
				\big\||u(s,\cdot)|^{\alpha_w}|v(s,\cdot)|^{\beta_w}-|\overline u(s,\cdot)|^{\alpha_w}|\overline v(s,\cdot)|^{\beta_w}\big\|_{L^{\gamma}}&\lesssim (1+s)^{-\frac{n}{2}\tron{\alpha_w+\beta_w-\frac{1}{\gamma}}}\\
				&\quad\times \|U-\overline{U}\|_{X(T)}\tron{\|U\|^{\alpha_w+\beta_w-1}_{X(T)}+\|\overline{U}\|^{\alpha_w+\beta_w-1}_{X(T)}}.
			\end{align*}
		\end{lemma}
		\begin{proof}
			By \eqref{con_G-N_Thr1}, we follow the same approach as in Lemma \ref{lemma_G-N_Thr1globalexistence} to conclude the first estimate. To prove the second one, we apply Lemma \ref{Lem_Lipschitz} combined with H\"older's inequality to obtain
			\begin{equation}\label{embedding_Lipschitz}
				\begin{aligned}
					&\big\||u(s,\cdot)|^{\alpha_w}|v(s,\cdot)|^{\beta_w}-|\overline u(s,\cdot)|^{\alpha_w}|\overline v(s,\cdot)|^{\beta_w}\big\|_{L^{\gamma}} \\
					&\quad\lesssim \|u(s,\cdot)-\overline u(s,\cdot)\|_{L^{\gamma(\alpha_w+\beta_w)}}\big(\|u(s,\cdot)\|_{L^{\gamma(\alpha_w+\beta_w)}}^{\alpha_w-1}+\|\overline u(s,\cdot)\|_{L^{\gamma(\alpha_w+\beta_w)}}^{\alpha_w-1}\big)\|v(s,\cdot)\|_{L^{\gamma(\alpha_w+\beta_w)}}^{\beta_w}\\
					&\qquad+\|v(s,\cdot)-\overline v(s,\cdot)\|_{L^{\gamma(\alpha_w+\beta_w)}}\big(\|v(s,\cdot)\|_{L^{\gamma(\alpha_w+\beta_w)}}^{\beta_w-1}+\|\overline v(s,\cdot)\|_{L^{\gamma(\alpha_w+\beta_w)}}^{\beta_w-1}\big)\|\overline u(s,\cdot)\|_{L^{\gamma(\alpha_w+\beta_w)}}^{\alpha_w}.
				\end{aligned}
			\end{equation}
			The employment of Lemma \ref{fractionalGagliardoNirenberg} gives
			\begin{equation}\label{embedding_u_j}
				\|\phi(s,\cdot)\|_{L^{\gamma(\alpha_w+\beta_w)}}\lesssim
				\|\phi(s,\cdot)\|_{L^r}^{1-\theta_{r,w}}
				\|\phi(s,\cdot)\|_{\dot{H}^1}^{\theta_{r,w}} \quad\text{ with } 
				\theta_{r,w}=\frac{\frac{1}{r}-\frac{1}{\gamma(\alpha_w+\beta_w)}}{\frac{1}{r}-\frac{1}{2}+\frac{1}{n}} \in [0,1]
			\end{equation}
			for $\phi\in\{u,v,\overline u,\overline v,u-\overline u,v-\overline v\}$. Combining \eqref{embedding_Lipschitz}, \eqref{embedding_u_j} and the norm definition in $X(T)$, we obtain the second desired estimate in Lemma \ref{lemma_G-N_Thr2globalexistence}. Hence, the proof of Lemma \ref{lemma_G-N_Thr2globalexistence} is established.
		\end{proof}
		
		Coming back to the proof of Proposition \ref{Thr_2_Globalexistence}, we use Lemma \ref{lemma_G-N_Thr2globalexistence} and Proposition \ref{LinearEstimates} with $(m_1,m_2)=(1,2)$, $(s_1, s_2) =(k,0)$ for $\tau \in [0, t/2]$, and $(m_1,m_2)=(2,2), (s_1,s_2) =(k,0)$ for $\tau \in (t/2,t]$ to obtain the following estimates for $k\in \{0,1\}$:
		\begin{align}
			\| w^{\rm nl}(t,\cdot)\|_{\dot{H}^k} &\lesssim \int_0^{t/2} (1+t-s)^{-\frac{n}{4}-\frac{k}{2}} \||u(s,x)|^{\alpha_w}|v(s,x)|^{\beta_w}\|_{L^1\cap L^2}ds\nonumber\\
			&\qquad + \int_{t/2}^t (1+t-s)^{-\frac{k}{2}} \||u(s,x)|^{\alpha_w}|v(s,x)|^{\beta_w}\|_{L^2}ds\nonumber\\
			&\lesssim  (1+t)^{-\frac{n}{4}-\frac{k}{2}} \|U\|_{X(T)}^{\alpha_w+\beta_w} \int_0^{t/2}(1+s)^{-\frac{n}{2}\tron{\alpha_w+\beta_w-1}}ds\nonumber\\
			&\qquad +  (1+t)^{-\frac{n}{2}\tron{\alpha_w+\beta_w-\frac{1}{2}}} \|U\|_{X(T)}^{\alpha_w+\beta_w} \int_{t/2}^t (1+t-s)^{-\frac{k}{2}}ds \nonumber\\
			&\lesssim (1+t)^{-\frac{n}{4}-\frac{k}{2}} \|U\|_{X(T)}^{\alpha_w+\beta_w}, \label{eq_H_2_Thr2} \\     
			\|\partial_t w^{\rm nl}(t,\cdot)\|_{L^2} &\lesssim \int_0^{t/2} (1+t-s)^{-\frac{n}{4}-1} \||u(s,x)|^{\alpha_w}|v(s,x)|^{\beta_w}\|_{L^1\cap L^2}ds\nonumber\\
			&\qquad + \int_{t/2}^t (1+t-s)^{-1} \||u(s,x)|^{\alpha_w}|v(s,x)|^{\beta_w}\|_{L^2}ds\nonumber\\
			&\lesssim  (1+t)^{-\frac{n}{4}-1} \|U\|_{X(T)}^{\alpha_w+\beta_w} \int_0^{t/2}(1+s)^{-\frac{n}{2}\tron{\alpha_w+\beta_w-1}}ds\nonumber\\
			&\qquad +  (1+t)^{-\frac{n}{2}\tron{\alpha_w+\beta_w-\frac{1}{2}}} \|U\|_{X(T)}^{\alpha_w+\beta_w} \int_{t/2}^t (1+t-s)^{-1}ds \nonumber\\
			&\lesssim (1+t)^{-\frac{n}{4}-1} \|U\|_{X(T)}^{\alpha_w+\beta_w} \label{eq_partial_t_Thr2}
		\end{align}
		for $w\in\{u,v\}$, where we used the assumption \eqref{Thr2globalexistence_Critical_exponent}. For the estimate of $L^r$-norm, we use lemma \ref{lemma_G-N_Thr2globalexistence} and Proposition \ref{LinearEstimates} with $(m_1,m_2)=(1,r)$, $(s_1, s_2) =(0,0)$ for $\tau \in [0, t/2]$, and $(m_1,m_2)=(r,r), (s_1,s_2) =(0,0)$ for $\tau \in (t/2,t]$ to obtain
		\begin{align}
			\| w^{\rm nl}(t,\cdot)\|_{L^r} &\lesssim \int_0^{t/2} (1+t-s)^{-\frac{n}{2}\tron{1-\frac{1}{r}}} \||u(s,x)|^{\alpha_w}|v(s,x)|^{\beta_w}\|_{L^1\cap L^r}ds\nonumber\\
			&\qquad + \int_{t/2}^t\||u(s,x)|^{\alpha_w}|v(s,x)|^{\beta_w}\|_{L^r}ds\nonumber \\
			&\lesssim  (1+t)^{-\frac{n}{2}\tron{1-\frac{1}{r}}} \|U\|_{X(T)}^{\alpha_w+\beta_w} \int_0^{t/2}(1+s)^{-\frac{n}{2}\tron{\alpha_w+\beta_w-1}}ds\nonumber\\
			&\qquad +  (1+t)^{-\frac{n}{2}\tron{\alpha_w+\beta_w-\frac{1}{r}}}\|U\|_{X(T)}^{\alpha_w+\beta_w} \int_{t/2}^t ds \nonumber\\
			&\lesssim (1+t)^{-\frac{n}{2}\tron{1-\frac{1}{r}}} \|U\|_{X(T)}^{\alpha_w+\beta_w}.
			\label{eq_L_r_Thr2}
		\end{align}
		Combining estimates \eqref{eq_H_2_Thr2}, \eqref{eq_partial_t_Thr2} and \eqref{eq_L_r_Thr2} with some estimates in Proposition \ref{LinearEstimates}, we obtain \eqref{eqNu_Thr2}. To indicate \eqref{eqNu-Nv_Thr2}, we apply Proposition \ref{LinearEstimates} with $(m_1,m_2) = (1,2)$, $(s_1, s_2) =(k,0)$ for $\tau \in [0, t/2]$, and $(m_1,m_2)=(2,2), (s_1,s_2) =(k,0)$ for $\tau \in (t/2, t]$ to arrive at
		\begin{align}
			\|w^{\rm nl}(t,\cdot)-\overline w^{\rm nl}(t,\cdot)\|_{\dot{H}^k} &\lesssim \int_{0}^{t/2}(1+t-s)^{-\frac{n}{4}-\frac{k}{2}}\||u(s,\cdot)|^{\alpha_w}|v(s,\cdot)|^{\beta_w}-|\overline u(s,\cdot)|^{\alpha_w}|\overline v(s,\cdot)|^{\beta_w}\|_{L^1\cap L^{2}}ds\nonumber\\
			&\quad+\int_{t/2}^t(1+t-s)^{-\frac{k}{2}}\||u(s,\cdot)|^{\alpha_w}|v(s,\cdot)|^{\beta_w}-|\overline u(s,\cdot)|^{\alpha_w}|\overline v(s,\cdot)|^{\beta_w}\|_{L^{2}}ds\nonumber\\
			&\lesssim (1+t)^{-\frac{n}{4}-\frac{k}{2}}\|U-\overline{U}\|_{X(T)}\tron{\|U\|^{\alpha_w+\beta_w-1}_{X(T)}+\|\overline{U}\|^{\alpha_w+\beta_w-1}_{X(T)}} \label{es_u_v_Lbeta_i}
		\end{align}
		and
		\begin{align}
			\|\partial_tw^{\rm nl}(t,\cdot)-\partial_t\overline w^{\rm nl}(t,\cdot)\|_{L^2} &\lesssim (1+t)^{-\frac{n}{4}-1}\|U-\overline{U}\|_{X(T)}\tron{\|U\|^{\alpha_w+\beta_w-1}_{X(T)}+\|\overline{U}\|^{\alpha_w+\beta_w-1}_{X(T)}} \label{es_u-v_partial_t_i}
		\end{align}
		for $w\in\{u,v\}$ and $k\in \{0,1\}$. Finally, we employ Proposition \ref{LinearEstimates} again with $(m_1,m_2) = (1,r)$, $(s_1, s_2) =(0,0)$ for $\tau \in [0, t/2]$, and $(m_1,m_2)=(r,r), (s_1,s_2) =(0,0)$ for $\tau \in (t/2, t]$ to estimate the $L^r$-norm
		\begin{align}
			\|w^{\rm nl}(t,\cdot)-\overline w^{\rm nl}(t,\cdot)\|_{L^r} &\lesssim \int_{0}^{t/2}(1+t-s)^{-\frac{n}{2}\tron{1-\frac{1}{r}}}\||u(s,\cdot)|^{\alpha_w}|v(s,\cdot)|^{\beta_w}-|\overline u(s,\cdot)|^{\alpha_w}|\overline v(s,\cdot)|^{\beta_w}\|_{L^1\cap L^{r}}ds\nonumber\\
			&\quad+\int_{t/2}^t\||u(s,\cdot)|^{\alpha_w}|v(s,\cdot)|^{\beta_w}-|\overline u(s,\cdot)|^{\alpha_w}|\overline v(s,\cdot)|^{\beta_w}\|_{L^{r}}ds\nonumber\\
			&\lesssim (1+t)^{-\frac{n}{2}\tron{1-\frac{1}{r}}}\|U-\overline{U}\|_{X(T)}\tron{\|U\|^{\alpha_w+\beta_w-1}_{X(T)}+\|\overline{U}\|^{\alpha_w+\beta_w-1}_{X(T)}}. \label{es_u_v_L_r}
		\end{align}
		Linking \eqref{es_u_v_Lbeta_i}, \eqref{es_u-v_partial_t_i} and \eqref{es_u_v_L_r} we may conclude \eqref{eqNu-Nv_Thr2} from the norm definition of $X(T)$. This completes the proof of Proposition \ref{Thr_2_Globalexistence}.
	\end{proof}
	
	\begin{remark}
		{\rm
			From the proof of Proposition \ref{Thr_2_Globalexistence}, one recognizes that if $p,q\geq 1$, then we have a unique global (in time) mild solution to \eqref{Main_system} coming from a application of Banach's fixed point theorem instead. In other words, this case can be considered an illustration to tell us about the essential difference when we want to apply two kinds of fixed point theorems.
		}
	\end{remark}
	
	\begin{proposition}\label{Thr2_Blowup}
		Let $p,q>0$ and $\alpha,\beta\in[1,\infty)$. Assume that $1\leq n\leq 3$, $u_0=v_0=0$ and $u_1,v_1\in L^1\cap H^1$ fulfilling $M_u>0$ and $M_v>0$, where two quantities $M_u$ and $M_v$ are defined as in Theorem \ref{Thr1_Blowup}. If
		\begin{equation}\label{Thr2_Blowup_Criticalexponent}
			\min\{p+\alpha,q+\beta\}<1+2/n,
		\end{equation}
		then there is no global (in time) mild solution to \eqref{Main_system} in the sense of Definition \ref{MildSol.Def}.
	\end{proposition}
	
	\begin{proof}
		We use the notation introduced in the proof of Theorem \ref{Thr1_Blowup}.
		It remains to show that both $a_k\to-\infty$ and $b_k\to-\infty$ as $k\to\infty$. Indeed, it is obvious to see that the sequences $\{a_k\}_{k\in \mathbb N}$ and $\{b_k\}_{k\in \mathbb N}$ are nonincreasing. Without loss of generality, we assume that $\mu:=n(p+\alpha-1)/2-1<0$ due to the assumption \eqref{Thr2_Blowup_Criticalexponent}. Hence, it holds
		\[
		a_{k+1}-a_k\leq (\alpha-1)a_k+pb_k-1\leq(\alpha-1)a_0+pb_0-1=\mu<0,
		\]
		which gives $a_{k}\leq a_0+k\mu$. Therefore, we conclude that $a_k\to -\infty$ as $k\to\infty$. Moreover, one observes
		$$ b_{k+1}\leq qa_k+\beta b_k-1\leq qa_k+\beta b_0-1 $$
		to arrive at $b_k\to -\infty$ as $k\to\infty$. Repeating the corresponding steps in the proof of Theorem \ref{Thr1_Blowup} we may claim the desired result.
	\end{proof}
	
	\begin{remark}
		{\rm
			From the conditions \eqref{Thr2globalexistence_Critical_exponent} and \eqref{Thr2_Blowup_Criticalexponent} in Propositions \ref{Thr_2_Globalexistence} and \ref{Thr2_Blowup}, respectively, one recognizes that the sharp threshold to separate existence and nonexistence of global mild solutions is given by 
			\begin{equation*}\label{Thr2_Criticalexponent}
				\min\{p+\alpha,q+\beta\}= 1+2/n.
			\end{equation*}
			In a comparison with \eqref{crit_curve_sublinear}, we can see that the critical threshold of the system \eqref{Main_system} shifts from the so-called critical curve, when $\alpha,\beta \in [0,1)$, to the so-called critical exponent, when $\alpha,\beta \in [1,\infty)$. This means that the former case is of a system-like threshold as \eqref{critical_curve_weakly_coupled}, meanwhile the latter case is of a single equation-like threshold only as the Fujita exponent $p_{\mathrm{Fuj}}$. From this observation, we understood the remarkable influence of parameters $\alpha,\beta$ on determining critical thresholds of \eqref{Main_system}.
		}
	\end{remark}
	
	\begin{remark}[\textbf{Open problems}]
		{\rm
			Let us discuss some of further interesting problems as follows:
			\begin{itemize}[leftmargin=*]
				\item[$\bullet$] The first one is to consider the fractional Laplacian $(-\Delta)^{\sigma}$ for any $\sigma>1$ instead. In this setting, the loss of the finite propagation speed makes dificult in terms of both the compactness argument in the proof of global existence and the localized iteration in the proof of a blow-up result.
				\item[$\bullet$] Another interesting problem arises when $\alpha$ and $\beta$ belong to different range, for example, $\alpha\in[0,1)$ and $\beta\in[1,\infty)$. Then, one nonlinear term is non-Lipschitz, whereas the other is locally Lipschitz. This means that it seems more delicate to determine the critical threshold since we do not know whether the nonlinearity term $|u|^{\alpha}|v|^{p}$ or $|u|^{q}|v|^{\beta}$ dominates the other.
				\item[$\bullet$] We would like to emphasize that the critical cases 
				\[
				\frac{1+\max\{q-\alpha,p-\beta\}}{pq-(\alpha-1)(\beta-1)}=\frac{n}{2} \quad\text{ or }\quad \min\{p+\alpha,q+\beta\}=1+2/n
				\]
				when $\alpha,\beta\in (0,1)$ or $\alpha,\beta\in [1,\infty)$, respectively, remain open since the present approach of power-type iteration may reach a fixed point and no longer improve the lower bounds. Probably, a logarithmic iteration associated with weighted lower estimates, or a refined test-function method may be required to study this problem.
			\end{itemize}
		}
	\end{remark}

	\section*{Acknowledgments}
	This research is funded by Vietnam National Foundation for Science and Technology Development (NAFOSTED) under grant number 101.02-2025.23. The authors would like to thank M.Sc. Trinh Dang Duong for helpful discussions during the preparation of this paper.
	
	\appendix 
	\section{Auxiliary lemmas}
	\begin{lemma}\label{Lem_Lipschitz}
		Let $w\in\{u,v\}$ and $\alpha_w,\beta_w\geq1$. Then, it holds
		\begin{align*}
			\big||u|^{\alpha_w}|v|^{\beta_w}-|\overline{u}|^{\alpha_w}|\overline{v}|^{\beta_w}\big|\lesssim |u-\overline{u}|\big(|u|^{\alpha_w-1}+|\overline{u}|^{\alpha_w-1}\big)|v|^{\beta_w}+|v-\overline{v}|\big(|v|^{\beta_w-1}+|\overline{v}|^{\beta_w-1}\big)|\overline{u}|^{\alpha_w}.
		\end{align*}
	\end{lemma}
	\begin{proof}
		It is clear to see that
		\begin{equation*}\label{eq_1_Lem_Lipschitz}
			|u|^{\alpha_w}|v|^{\beta_w}-|\overline{u}|^{\alpha_w}|\overline{v}|^{\beta_w}=\left(|u|^{\alpha_w}-|\overline{u}|^{\alpha_w}\right)|v|^{\beta_w}+|\overline{u}|^{\alpha_w}\big(|v|^{\beta_w}-|\overline{v}|^{\beta_w}\big).
		\end{equation*}
		By the mean value theorem, we obtain
		\begin{align*}
			\big||u|^{\alpha_w}-|\overline{u}|^{\alpha_w}\big|\lesssim |u-\overline{u}|\left(|u|^{\alpha_w-1}+|\overline{u}|^{\alpha_w-1}\right),\quad\big||v|^{\beta_w}-|\overline{v}|^{\beta_w}\big|\lesssim |v-\overline{v}|\big(|v|^{\beta_w-1}+|\overline{v}|^{\beta_w-1}\big).
		\end{align*}
		Hence, combining these above inequalities we deduce the desired estimate.
	\end{proof}
	\begin{lemma}\label{regularity_K}
		Let $T>0$ and $g\in L^1([0,T];L^2)$. We define
		\[
		v(t,\cdot):=\int_0^t\mathcal K(t-s,\cdot)\ast g(s,\cdot)\,ds,
		\qquad t\in[0,T].
		\]
		Then, $v\in \mathcal{C}([0,T];H^1)\cap \mathcal{C}^1([0,T];L^2)$ and $v$ is a weak solution to 
		\[
		\left\{
		\begin{aligned}
			&v_{tt}-\Delta v+ v_t= g,&&x\in \mathbb R^n,t\in (0,T),\\
			&v(0,x)=0, \quad v_t(0,x)=0,&&x\in \mathbb R^n.
		\end{aligned}
		\right.
		\]
		Furthermore, for any
		$\Phi\in \mathcal{C}^\infty_{\rm c}([0,T)\times\mathbb R^n)$ it holds
		\begin{equation}\label{duhamel_distribution_identity}
			\int_0^T\int_{\mathbb R^n}\Bigl(-\partial_tv\,\partial_t\Phi+\nabla v\cdot\nabla\Phi+\partial_tv\,\Phi- g\,\Phi\Bigr)\,dx\,dt=0
		\end{equation}
		and we have the following estimate:
		\begin{equation}\label{Energy_Est}
			\sup_{0\leq t\leq T}\big\{\|\partial_tv(t,\cdot)\|_{L^2}+\|\nabla v(t,\cdot)\|_{L^2}\big\}\lesssim \|g\|_{L^1([0,T];L^2)}.
		\end{equation}
	\end{lemma}
	\begin{proof}
		For $h\in L^2$, set $w(t,\cdot)=\mathcal K(t,\cdot)\ast h(\cdot)$ and $\partial_tw(t,\cdot)=\partial_t\mathcal K(t,\cdot)\ast h(\cdot)$. By the definition of $\mathcal K(t,\cdot)$ and the formula of Parseval-Plancherel, it holds
		\begin{equation}\label{uniformbound_H_1_and_t}
			\sup_{0\leq t\leq T}\big\{\|\partial_tw(t,\cdot)\|_{L^2}+\|\nabla w(t,\cdot)\|_{L^2}\big\}\lesssim \|h(\cdot)\|_{L^2}.
		\end{equation}
		Since $w(0,\cdot)=0$, we get $\|w(t,\cdot)\|_{L^2}\leq t\|h(\cdot)\|_{L^2}$. Thus, for any $T>0$ one has
		\begin{equation}\label{regularity_K_eq1}
			\sup_{0\leq t\leq T}\big\{\|w(t,\cdot)\|_{H^1}+\|\partial_tw(t,\cdot)\|_{L^2}
			\big\}\leq C(T)\|h(\cdot)\|_{L^2}.
		\end{equation}
		Moreover, the Fourier multipliers of $\mathcal{K}(t,\cdot)$ and $\partial_t K(t,\cdot)$ are continuous in $t$. So, we conclude that
		\begin{equation}\label{regularity_K_eq2}
			w\in \mathcal{C}([0,T];H^1) \quad \text{ and }\quad
			\partial_tw \in \mathcal{C}([0,T];L^2).
		\end{equation}
		We now write
		\[
		v(t,\cdot)=\int_0^t\mathcal{K}(t-s,\cdot)\ast g(s,\cdot)\,ds.
		\]
		In the view of \eqref{regularity_K_eq1}, this Bochner integral is well-defined and we may conclude the estimate \eqref{Energy_Est}. Let us next prove \eqref{duhamel_distribution_identity}. By \eqref{regularity_K_eq1}, we also have
		\[
		\|v(t,\cdot)\|_{H^1}
		\leq C(T)\int_0^t \|g(s,\cdot)\|_{L^2}\,ds.
		\]
		We next prove the continuity of $v$. Let $\tau>0$ be such that $t+\tau\leq T$. Then, it holds
		\[
		\begin{aligned}
			v(t+\tau,\cdot)-v(t,\cdot) &=\int_0^t
			\bigl(\mathcal{K}(t+\tau-s,\cdot)-\mathcal{K}(t-s,\cdot)\bigr)\ast g(s,\cdot)\,ds \\
			&\qquad +\int_t^{t+\tau}\mathcal{K}(t+\tau-s,\cdot)\ast g(s,\cdot)\,ds.
		\end{aligned}
		\]
		For a.e $s\in(0,t)$, the first integrand converges to zero in $H^1$ as $\tau\to0$, by
		\eqref{regularity_K_eq2}. Moreover, by
		\eqref{regularity_K_eq1}, it is bounded by $2C(T)\|g(s,\cdot)\|_{L^2}$, which belongs to $L^1(0,t)$. Hence, the first integral tends to zero by the dominated convergence theorem. The second integral satisfies
		\[
		\left\|\int_t^{t+\tau}\mathcal{K}(t+\tau-s,\cdot)\ast g(s,\cdot)\,ds\right\|_{H^1}\leq C(T)\int_t^{t+\tau}\| g(s,\cdot)\|_{L^2}\,ds,
		\]
		which tends to zero as $\tau \to 0$. The case $\tau<0$ is analogous. Therefore, $v\in \mathcal{C}([0,T];H^1)$. Taking the time-derivative of the Bochner convolution gives
		\begin{equation}\label{regularity_K_eq4}
			\partial_tv(t,\cdot)=\int_0^t\partial_t \mathcal{K}(t-s,\cdot)\ast g(s,\cdot)\,ds.
		\end{equation}
		Indeed, we have the following convergence for the first term 
		\[
		\int_0^t
		\frac{\bigl(\mathcal{K}(t+\tau-s,\cdot)-\mathcal{K}(t-s,\cdot)\bigr)\ast g(s,\cdot)}{\tau}\,ds\to \int_0^t\partial_t \mathcal{K}(t-s,\cdot)\ast g(s,\cdot)\,ds\text{ in }L^2
		\]
		as $\tau\to 0$ due to the dominated convergence theorem while the second term vanishes because
		\[
		\left\|\frac{1}{\tau}\int_t^{t+\tau}\mathcal{K}(t+\tau-s,\cdot)\ast g(s,\cdot)\,ds \right\|_{L^2}\leq \int_t^{t+\tau}\|g(s,\cdot)\|_{L^2}\,ds\to 0
		\]
		as $\tau\to 0$. Finally, by using \eqref{regularity_K_eq4}, for $\tau>0$ we get
		\[
		\begin{aligned}
			\partial_tv(t+\tau,\cdot)-\partial_tv(t,\cdot)&=\int_0^t
			\bigl(\partial_t \mathcal{K}(t+\tau-s,\cdot)-\partial_t\mathcal{K}(t-s,\cdot)\bigr)\ast g(s,\cdot)\,ds\\
			&\qquad +\int_t^{t+\tau}\partial_t \mathcal{K}(t+\tau-s,\cdot)\ast g(s,\cdot)\,ds.
		\end{aligned}
		\]
		The first integral tends to zero in $L^2$ by \eqref{regularity_K_eq1}, \eqref{regularity_K_eq2}, and the dominated convergence theorem, whereas
		\[
		\left\|\int_t^{t+\tau}\partial_t\mathcal{K}(t+h-s,\cdot)g(s,\cdot)\,ds\right\|_{L^2}\leq C(T)\int_t^{t+\tau}\| g(s,\cdot)\|_{L^2}\,ds\to 0.
		\]
		Thus, $\partial_tv\in \mathcal{C}([0,T];L^2)$. In particular, the formulas above give $v(0,\cdot)=0, \partial_tv(0,\cdot)=0$. It remains to verify the distributional identity. After choosing the function $g_\ell\in \mathcal{C}^\infty_{\rm c}((0,T)\times\mathbb R^n)$ such that $g_\ell\to g$ in $L^1([0,T];L^2)$ as $\ell\to\ity$, we define
		\[
		v_\ell(t,\cdot):=\int_0^t\mathcal{K}(t-s,\cdot)g_\ell(s,\cdot)\,ds.
		\]
		For any $\ell\in\mathbb N$, $v_\ell\in \mathcal{C}([0,T],H^2)\cap \mathcal{C}^1([0,T],H^1)\cap \mathcal{C}^2([0,T],L^2)$ is a strong solution to
		\[
		\left\{
		\begin{aligned}
			&\partial_t^2v_\ell-\Delta v_\ell+\partial_tv_\ell=g_\ell,&&x\in \mathbb{R}^n,t>0,\\
			&v_\ell(0,\cdot)=0,\quad \partial_tv_\ell(0,\cdot)=0,&&x\in \mathbb{R}^n.
		\end{aligned}
		\right.
		\]
		Hence, for any $\Phi\in \mathcal{C}^\infty_{\rm c}([0,T)\times\mathbb R^n)$, the integral-by-part fomurla implies that
		\begin{equation}\label{weak-identity}
			\int_0^T\int_{\mathbb R^n}\Bigl(-\partial_tv_\ell\,\partial_t\Phi+\nabla v_\ell\cdot\nabla\Phi+\partial_tv_\ell\,\Phi-g_\ell\Phi\Bigr)\,dx\,dt=0.
		\end{equation}
		By \eqref{regularity_K_eq1} and \eqref{regularity_K_eq4}, we arrive at
		\[
		\sup_{0\leq t\leq T}\Bigl(\|v_\ell(t,\cdot)-v(t,\cdot)\|_{H^1}+\|\partial_tv_\ell(t,\cdot)-\partial_tv(t,\cdot)\|_{L^2} \Bigr)\leq C(T)\|g_\ell- g\|_{L^1([0,T];L^2)} \to 0
		\]
		as $\ell\to\ity$. Passing to the limit in the preceding weak identity \eqref{weak-identity} yields \eqref{duhamel_distribution_identity}. Therefore, $v$ is a distributional solution with zero initial data. This completes the proof.
	\end{proof}
	\begin{lemma}[see \cite{Hajaiej2011}]  \label{fractionalGagliardoNirenberg}
		Let $1<r,\,r_0,\,r_1<\infty$, $a >0$ and $\sigma\in [0,a)$. Then, it holds
		$$ \|u\|_{\dot{H}^{\sigma}_r} \lesssim \|u\|_{L^{r_0}}^{1-\omega(\sigma,a)}\, \|u\|_{\dot{H}^{a}_{r_1}}^{\omega(\sigma,a)}, \text{ where }\omega(\sigma,a) := \dfrac{\frac{1}{r_0}-\frac{1}{r}+\frac{\sigma}{n}}{\frac{1}{r_0}-\frac{1}{r_1}+\frac{a}{n}} \in [\sigma/a,1]. $$
	\end{lemma}

\end{document}